\documentclass[oneside,english]{amsart}
\usepackage[T1]{fontenc}
\usepackage[latin9]{inputenc}
\usepackage{geometry}
\usepackage{color}
\usepackage{babel}
\usepackage{verbatim}
\usepackage{mathrsfs}
\usepackage{mathtools}
\usepackage{bm}
\usepackage{amsthm}
\usepackage{amstext}
\usepackage{amssymb}
\usepackage{esint}
\usepackage[all]{xy}
\usepackage[unicode=true,pdfusetitle,
 bookmarks=true,bookmarksnumbered=false,bookmarksopen=false,
 breaklinks=false,pdfborder={0 0 0},backref=false,colorlinks=true]
 {hyperref}

\makeatletter
\numberwithin{equation}{section}
\numberwithin{figure}{section}
  \theoremstyle{plain}
  \newtheorem{lem}{\protect\lemmaname}
  \theoremstyle{remark}
  \newtheorem{rem}{\protect\remarkname}
\theoremstyle{plain}
\newtheorem{thm}{\protect\theoremname}
  \theoremstyle{plain}
  \newtheorem{cor}{\protect\corollaryname}
  \theoremstyle{plain}
  \newtheorem{prop}{\protect\propositionname}
 \theoremstyle{definition}
  \newtheorem{example}{\protect\examplename}

\DeclareMathOperator{\GL}{GL}

\usepackage{bm}

\IfFileExists{lmodern.sty}{\usepackage{lmodern}}{}

\newcommand{\xyR}[1]{%
\xydef@\xymatrixrowsep@{#1}}

\newcommand{\xyC}[1]{%
\xydef@\xymatrixcolsep@{#1}}

\makeatother

  \providecommand{\examplename}{Example}
  \providecommand{\lemmaname}{Lemma}
  \providecommand{\propositionname}{Proposition}
  \providecommand{\remarkname}{Remark}
\providecommand{\corollaryname}{Corollary}
\providecommand{\theoremname}{Theorem}

\begin{document}
\global\long\def\integral#1#2{{\displaystyle \intop_{#1}^{#2}}}

\title[Finite Dimensional Fokker-Planck Equations For CTRWLs]{Finite Dimensional Fokker-Planck Equations For  Continuous
Time Random Walk Limits}

\author{Ofer Busani}
\begin{abstract}
Continuous Time Random Walk(CTRW) is a model where particle's jumps
in space are coupled with waiting times before each jump. A Continuous
Time Random Walk Limit(CTRWL) is obtained by a limit procedure on
a CTRW and can be used to model anomalous diffusion. The distribution
$p\left(dx,t\right)$ of a CTRWL $X_{t}$ satisfies a Fractional Fokker-Planck
Equation(FFPE). Since CTRWLs are usually not Markovian, their one
dimensional FFPE is not enough to completely define them. In this
paper we find the FFPEs of the distribution of $X_{t}$ at multiple
times , i.e. the distribution of the random vector $\left(X_{t_{1}},...,X_{t_{n}}\right)$
for $t_{1}<...<t_{n}$ for a large class of CTRWLs. This allows us
to define CTRWLs by their finite dimensional FFPEs. 
\end{abstract}

\maketitle

\section{Introduction}

CTRW models the movement of a particle in space, where the $k\mbox{'th}$
jump $J_{k}$ of the particle in space succeeds the $k\mbox{'th}$
waiting time $W_{k}$. We let $N_{t}=\sup\left\{ k:T_{k}\leq t\right\} $
where $T_{k}=\sum_{i=1}^{k}W_{i}$, if $T_{1}>t$ then we set $N_{t}$
to be $0$. $N_{t}$ is just the number of jumps of the particle up
to time $t$. Then
\[
X'_{t}=\sum_{k=1}^{N_{t}}J_{k},
\]
is the CTRW associated with the time-space jumps $\left\{ \left(J_{k},W_{k}\right)\right\} _{k\in\mathbb{N}}$.
Let us now assume that $\left\{ J_{k}\right\} $ and $\left\{ W_{k}\right\} $
are independent i.i.d sequences of random variables. In order to model
the long time behavior of the CTRW we write $\left\{ \left(J_{k}^{c},W_{k}^{c}\right)\right\} _{k\in\mathbb{N}}$
for $c>0$. Here the purpose of $c$ is to facilitate the convergence
of the trajectories of $\left\{ \left(J_{k}^{c},W_{k}^{c}\right)\right\} _{k\in\mathbb{N}}$
weakly on a proper space. More precisely, we let $\mathbb{\mathcal{D}}\left([0,\infty),\mathbb{R}^{2}\right)$
be the space of c�dl�g functions $f:[0,\infty)\rightarrow\mathbb{R}^{2}$
equipped with the Skorokhod $J_{1}$ topology. We assume that
\[
\left(S_{u}^{c},T_{u}^{c}\right)=\sum_{k=1}^{\left\lfloor cu\right\rfloor }\left(J_{k}^{c},W_{k}^{c}\right)\Rightarrow\left(A_{u},D_{u}\right)\qquad c\rightarrow\infty,
\]
where $\Rightarrow$ denotes weak convergence of measures with respect
to the $J_{1}$ topology. We further assume that the processes $A_{t}$
and $D_{t}$ are independent L\'{e}vy processes and that $D_{t}$ is a
strictly increasing subordinator. Denote by $X_{t}^{c}$ the CTRW
associated with $\left\{ \left(J_{k}^{c},W_{k}^{c}\right)\right\} _{k\in\mathbb{N}}$.
We then have (\cite[Theorem 3.6]{straka2011} and \cite[Lemma 2.4.5]{Straka2011diss})
\begin{equation}
X_{t}^{c}\Rightarrow X_{t}=A_{E_{t}}\qquad c\rightarrow\infty,\label{eq:convergence of CTRW}
\end{equation}
where $E_{t}=\inf\{s:D_{s}>t\}$ is the inverse of $D_{t}$ and $\Rightarrow$
means weak convergence on $\mathbb{\mathcal{D}}\left([0,\infty),\mathbb{R}\right)$
equipped with the $J_{1}$ topology. It is well known that $X_{t}$
is usually not Markovian, a fact that makes the task of finding basic
properties of $X_{t}$ nontrivial. One such task is finding the finite
dimensional distributions(FDDs) of the process $X_{t}$, i.e. $P\left(X_{t_{1}}\in dx_{1},...,X_{t_{n}}\in dx_{n}\right)$.
In \cite{Meerschaert}, Meerschaert and Straka used a semi-Markov
approach to find the FDDs for a large class of CTRWL. It turns out
that the discrete regeneration times of $X_{t}^{c}$ converge to a
set of points where $X_{t}$ is renewed. Once we know the next time
of regeneration of $X_{t}$, we no longer need older observations
in order to determine the future behavior of $X_{t}$. More mathematically,
denote by $R_{t}=D_{E_{t}}-t$ the time left before regeneration of
$X_{t}$ then $\left(X_{t},R_{t}\right)$ is a Markov process. One
can then use the transition probabilities of $\left(X_{t},R_{t}\right)$
along with the Chapman-Kolmogorov Equations in order to find $P\left(X_{t_{1}}\in dx_{1},...,X_{t_{n}}\in dx_{n}\right)$
for $t_{1}<...<t_{n}$ and $n\in\mathbb{N}$. This method was used
in \cite{busani2014} in order to find the FDD of the aged process
$X_{t}^{t_{0}}=X_{t}-X_{t_{0}}$. It is well known (\cite[Section 4.5]{Meerschaert2011a})
that the one dimensional distribution $p\left(dx,t\right)=P\left(X_{t}\in dx\right)$
satisfies a FFPE. Once again, as $X_{t}$ is non Markovian the FFPE
satisfied by $p\left(dx,t\right)$ is not enough to fully describe
$X_{t}$(as it does when $X_{t}$ is Markovian). Hence, a dual problem
to finding the FDDs is that of finding the finite dimensional FFPEs
of the FDDs of $X_{t}$. In this paper we obtain the finite dimensional
FFPEs for a large class of CTRWL. The results generalize the well
known one dimensional FFPE of CTRW(\cite{Meerschaert2004}) as well
as results in the finite dimensional case(\cite{baule2005joint},\cite{baule2007fractional}).
\\
In Section \ref{sec:Mathematical-Background} we present relevant
mathematical background for this paper and prepare the way for our
main result. It is divided into 4 subsections; Subsection \ref{sub:Notations}
introduces the notation to be used throughout the paper, Subsection
\ref{sub:Caputo-and-Riemann-Liouville} presents the Caputo and Riemann-Liouville
fractional derivatives, Subsection \ref{sub:Pseudo-differential-operators-of}
establishes results regarding pseudo-differential operators(PDOs)
on certain multivariable functions which facilitate the proof of Theorem
\ref{thm:governing equation of multi variable h} and Subsection \ref{sub:The-Semi-Markov-Approach}
presents briefly the work in \cite{Meerschaert} upon which we establish
our results.\\
Section \ref{sec:Fokker-Planck-Equations} presents our main results;
Theorem \ref{thm:governing equation of multi variable h} gives the
finite dimensional FFPEs of $E_{t}$, Corollary \ref{cor: MD FPE of A_E_t}
states the finite dimensional FFPEs of the process $X_{t}=A_{E_{t}}$
where the outer process $A_{t}$ and the subordinator $D_{t}$ are
independent. Finally, Corollary \ref{cor:MD FPE of general A_E_t}
gives the finite dimensional FFPEs of the coupled case. Section \ref{sec:Examples}
compares our results with the well known finite dimensional case.\\
In Section \ref{sec:Pseudo-Differential-Operators-on} we show that
if $\xi\left(-k,s\right)$ is the symbol of a PDO on a suitable Banach
space then $\xi\left(-\sum_{i=1}^{n}k_{i},\sum_{i=1}^{n}s_{i}\right)$
is also a symbol of a PDO on another Banach space. This complements
the results in Section \ref{sec:Fokker-Planck-Equations}.

\section{Mathematical Background\label{sec:Mathematical-Background}}

\subsection{\label{sub:Notations}Notations}

A well known method of solving partial differential equations of distributions
$p\left(dx_{1},...,dx_{n};t_{1},...,t_{n}\right)$ on $\mathbb{R}^{n}$
is taking the Fourier Transform(FT) of the distribution with respect
to the spatial variables and then the Laplace Transform(LT) with respect
to the time variables. This is referred to as the Fourier Laplace
Transform(FLT) of $p\left(dx_{1},...,dx_{n};t_{1},...,t_{n}\right)$.
More generally, for $m,n\in\mathbb{N}$ let $f$$\left(dx_{1},...,dx_{m};t_{1},...,t_{n}\right)$
be a finite measure on $\mathbb{R}^{m}$ for every $\mathbf{t}=\left(t_{1},...,t_{n}\right)$
s.t $0<t_{1}\leq\cdots\leq t_{n}$ and assume that $\integral{\mathbf{x}\in A}{}f\left(dx_{1},...,dx_{m};t_{1},...,t_{n}\right)$
is measurable as a function of $\mathbf{t}$ for every measurable
$A\subset\mathbb{R}^{m}$. We denote the FT of $f$ by 
\[
\widetilde{f}\left(k_{1},...,k_{m};t_{1},...,t_{n}\right)=\integral{x_{1}\in\mathbb{R}}{}\cdots\integral{x_{m}\in\mathbb{R}}{}e^{-i\sum_{j=1}^{m}k_{j}x_{j}}f\left(dx_{1},...,dx_{m};t_{1},...,t_{n}\right).
\]
When $f$ has density $f\left(x_{1},...,x_{m};t_{1},...,t_{n}\right)$
we denote the LT of $f$ by 
\[
\hat{f}\left(x_{1},...,x_{m};s_{1},...,s_{n}\right)=\integral{t_{1}=0}{\infty}\cdots\integral{t_{n}=0}{\infty}e^{-\sum_{j=1}^{n}s_{j}t_{j}}f\left(x_{1},...,x_{m};t_{1},...,t_{n}\right)dt_{1}\cdots dt_{n}.
\]
The FLT of $f$ is 
\[
\bar{f}\left(k_{1},...,k_{m};s_{1},...,s_{n}\right)=\integral{t_{1}=0}{\infty}\cdots\integral{t_{n}=0}{\infty}\integral{x_{1}\in\mathbb{R}}{}\cdots\integral{x_{n}\in\mathbb{R}}{}e^{-i\sum_{j=1}^{m}k_{j}x_{j}-\sum_{j=1}^{n}s_{j}t_{j}}f\left(dx_{1},...,dx_{m};t_{1},...,t_{n}\right)dt_{1}\cdots dt_{n}.
\]
We also denote by $\tilde{f}$ the FT of $f$ with respect to some
of its spatial variables, therefore, $\widetilde{f}\left(dx_{1},k_{2};t_{1},t_{2}\right)$
is the FT of $f$ w.r.t $x_{2}$. Similarly, $\hat{f}\left(dx_{1},dx_{2};s_{1},t_{2}\right)$
is the LT of $f$ w.r.t $t_{1}$ and $\bar{f}\left(k_{1},dx_{2};s_{1},t_{2}\right)$
is the FLT of $f$ w.r.t $x_{1}$ and $t_{1}$. When using the hat
symbol is cumbersome we also use $\hat{f}=\mathscr{L}\left(f\right)$.We
occasionally use bold font to represent the vector $\mathbf{x}=\left(x_{1},...,x_{n}\right)$
where the size of the vector is clear.

\subsection{\label{sub:Caputo-and-Riemann-Liouville}Caputo and Riemann-Liouville
Fractional Derivatives}

The Riemann-Liouville(RL) fractional derivative of index $0<\alpha<1$
is given by
\begin{equation}
\mathfrak{D}_{t}^{\alpha}f\left(t\right)=\frac{\partial}{\partial t}\frac{1}{\Gamma\left(1-\alpha\right)}\integral 0t\left(t-r\right)^{-\alpha}f\left(r\right)dr,\label{eq:RL fractional derivative}
\end{equation}
for a suitable function $f$ defined on $\mathbb{R}_{+}$. When the
variable with respect to which we take the derivative is obvious we
drop the subscript and just write $\mathbb{\mathfrak{D}}^{\alpha}f\left(t\right)$.
It can be verified that the LT of (\ref{eq:RL fractional derivative})
is
\[
\widehat{\mathbb{\mathfrak{D}}^{\alpha}f}\left(s\right)=s^{\alpha}\widehat{f}\left(s\right).
\]
Hence, the RL derivative is a PDO of symbol $s^{\alpha}$. Caputo's
derivative is obtained by moving the derivative in (\ref{eq:RL fractional derivative})
under the integral to obtain
\begin{equation}
\mathbb{D}_{t}^{\alpha}f\left(t\right)=\frac{1}{\Gamma\left(1-\alpha\right)}\integral 0t\left(t-r\right)^{-\alpha}\frac{\partial}{\partial r}f\left(r\right)dr.\label{eq:Caputo's fractional derivative}
\end{equation}
The LT of (\ref{eq:Caputo's fractional derivative}) is 
\[
\widehat{\mathfrak{\mathbb{D}}^{\alpha}f}\left(s\right)=s^{\alpha}\widehat{f}\left(s\right)-s^{\alpha-1}f\left(0^{+}\right).
\]
We denote the classic derivative by $\frac{\partial}{\partial t}=\mathbb{D}^{1}$,
and note that $\mathbb{D}^{1}=\mathfrak{D}^{1}$ iff $f\left(0^{+}\right)=0$.
For simplicity we drop the superscript and write $\frac{\partial}{\partial t}=\mathbb{\mathcal{\mathbb{D}}}$(or
$\frac{\partial}{\partial t}=\mathfrak{D}$ when that is the case).

\subsection{\label{sub:Pseudo-differential-operators-of}Pseudo-differential
operators of multivariable functions}

Here we investigate the PDOs acting on measures $f\left(dx_{1},...,dx_{n}\right)$
on $\mathbb{R}_{+}^{n}$ with support in $A^{n}=\left\{ \mathbf{x}:0\leq x_{1}\leq x_{2}\leq...\leq x_{n}\right\} $
with LT $\hat{f}$. Let $\mathbf{k}=\left(k_{1},...,k_{l}\right)$
be a strictly increasing l-tuple where $1\leq k_{i}\leq n$ for $1\leq i\leq l\leq n$
and s.t $k_{1}=1$. We shall sometimes abuse notation and write $i\in\mathbf{k}$
where we mean that $i=k_{j}$ for some $1\leq j\leq l$ .We also write
$\mathbf{k}^{c}$ for the increasing vector s.t $i\in\mathbf{k}^{c}$
iff $2\leq i\leq n$ and $i\notin\mathbf{k}$. If \textbf{$\mathbf{x}$}
is a vector of length $n$ we write $\mathbf{x_{\mathbf{k}}}$ for
the vector of length $l$ whose i'th element is $x_{k_{i}}$. Let
$A_{\mathbf{k}}^{n}$ be the set of all $\mathbf{x}\in A^{n}$ s.t
$x_{i-1}<x_{i}$ iff $i\in\mathbf{k}$ and where $x_{0}=0$. For example,
for $n=3$ $A_{\left(1,2\right)}^{3}=\left\{ \mathbf{x}:0<x_{1}<x_{2}=x_{3}\right\} $.
Since our interest in these distributions comes from the FDDs of the
process $E_{t}$, i.e. $h\left(dx_{1},...,dx_{n};t_{1},...,t_{n}\right)=P\left(E_{t_{1}}\in dx_{1},...,E_{t_{n}}\in dx_{n}\right)$
we also assume in this subsection that $f\left(dx_{1},...,dx_{n}\right)$
can be written as $f\left(dx_{1},...,dx_{n}\right)=f\left(x_{k_{1}},...,x_{k_{l}}\right)\delta_{k_{1}^{c}-1}\left(dx_{k_{1}^{c}}\right)\times...\times\delta_{k_{n-l}^{c}-1}\left(dx_{k_{n-l}^{c}}\right)dx_{k_{1}}\cdots dx_{k_{l}}$
where $f\left(x_{k_{1}},...,x_{k_{l}}\right)$ is absolutely continuous(a.c)
in in each of its variables, i.e. $x_{k_{i}}\rightarrow f\left(x_{1},...,x_{n}\right)$
is a.c with respect to Lebesgue measure on $\mathbb{R}$ for each
$1\leq i\leq l$. We occasionally refer to such $f$ as a.c, not to
be confused with the concept of a.c measure on $\mathbb{R}^{n}$.
We abbreviate by writing 
\begin{equation}
f_{\mathbf{k}}\left(d\mathbf{x}\right)=f\left(\mathbf{x}_{\mathbf{k}}\right)\delta_{\mathbf{k}^{c}-1}\left(d\mathbf{x}_{\mathbf{k}^{c}}\right)d\mathbf{x}_{\mathbf{k}},\label{eq:Abrr. form}
\end{equation}
so that $\mathbf{k}$ points out the indices for which $f$ has absolutely
continuous density. For example, $f_{\left(1,2,4\right)}\left(dx_{1},dx_{2},dx_{3},dx_{4}\right)$
can be written as $f\left(x_{1},x_{2},x_{4}\right)\delta_{x_{2}}\left(dx_{3}\right)dx_{1}dx_{2}dx_{4}$.
To motivate this assumption cf. (\ref{eq:transition of H}) and note
that by (\ref{eq:Finite dimensional distribution general}) $h\left(dx_{1},...,dx_{n};t_{1},...,t_{n}\right)$
is of the form $f\left(\mathbf{x}_{\mathbf{k}}\right)$ on $A_{\mathbf{k}}^{n}$.
The set $A_{\mathbf{k}}^{n}$ is a manifold of dimension $l$, and
represents the event where the process $E_{t}$ has been stuck at
the point $x_{i}$ since the time $t_{i-1}$ to $t_{i}$ for $i\notin\mathbf{k}$.
For example, $A_{\left(1,3\right)}^{4}$ represents the event $\left\{ E_{t_{1}}=x_{1}\in\left(0,\infty\right),E_{t_{2}}=x_{1},E_{t_{3}}=x_{3}\in\left(x_{1},\infty\right),E_{t_{4}}=x_{3}\right\} $,
and it helps to think of $\mathbf{k}$ as the indices of mobilized
points of the particle. Let us define a derivative operator on $f_{\mathbf{k}}$
distributions. We define the derivative operator to be
\[
\mathbb{D}_{\mathbf{x}}f_{\mathbf{k}}\left(d\mathbf{x}\right)=\sum_{i=1}^{l}\frac{\partial}{\partial x_{k_{i}}}f\left(\mathbf{x}_{\mathbf{k}}\right)\delta_{\mathbf{k}^{c}-1}\left(d\mathbf{x}_{\mathbf{k}^{c}}\right)d\mathbf{x}_{\mathbf{k}}.
\]
For example, if $f_{\left(1,2\right)}\left(d\mathbf{x}\right)=f\left(x_{1},x_{2}\right)\delta_{x_{2}}\left(dx_{3}\right)dx_{1}dx_{2}$
then 
\begin{alignat*}{1}
\mathfrak{\mathbb{D}}_{\mathbf{x}}f_{\left(1,2\right)}\left(d\mathbf{x}\right) & =\frac{\partial}{\partial x_{1}}f_{\left(1,2\right)}\left(x_{1},x_{2}\right)\delta_{x_{2}}\left(dx_{3}\right)dx_{1}dx_{2}\\
 & +\frac{\partial}{\partial x_{2}}f_{\left(1,2\right)}\left(x_{1},x_{2}\right)\delta_{x_{2}}\left(dx_{3}\right)dx_{1}dx_{2}.
\end{alignat*}
Note that $\mathfrak{D}_{\mathbf{x}}$ is well defined as we assume
that $f_{\mathbf{k}}$ has a.c density in $x_{i}$ for $i\in\mathbf{k}$.
We also assume that $\underset{x_{k_{l}}\rightarrow\infty}{\lim}e^{-x_{k_{l}}}f\left(x_{k_{1}},\ldots,x_{k_{l}}\right)=0$
where $f$ is as in (\ref{eq:Abrr. form}). This is not a strong assumption
as $f$ has LT. 
\begin{lem}
\label{lem:LT of the derivative of multivariable function}Let $f_{\mathbf{k}}$
be such that $l=n$. Then the LT of $\mathfrak{\mathbb{D}}_{\mathbf{x}}f\left(\mathbf{x}\right)$
is 
\begin{equation}
\widehat{\mathfrak{\mathbb{D}}_{\mathbf{x}}f_{\mathbf{k}}}\left(\mathbf{s}\right)=\left(\sum_{i=1}^{n}s_{i}\right)\widehat{f}_{\mathbf{k}}\left(s_{1},\ldots s_{n}\right)-\underset{x_{1}\rightarrow0^{+}}{\lim}\hat{f_{\mathbf{k}}}\left(x_{1},s_{2},...,s_{n}\right).\label{eq:LT of the derivative of multi variable function}
\end{equation}
 \end{lem}
\begin{proof}
In the following, we use $\check{a_{i}}$ to indicate that $a_{i}$
is absent from where it normally should be. Since here $f_{\mathbf{k}}\left(d\mathbf{x}\right)=f\left(\mathbf{x}\right)d\mathbf{x}$,
for $1\leq i\leq n$ we have
\begin{alignat}{1}
 & \integral{x_{1}=0}{\infty}\integral{x_{2}=0}{\infty}\cdots\integral{x_{n}=0}{\infty}e^{-\mathbf{\left\langle s,\mathbf{x}\right\rangle }}\frac{\partial f\left(\mathbf{x}\right)}{\partial x_{i}}d\mathbf{x}\label{eq:LT of multivariable function drivative}\\
 & =\integral{x_{1}=0}{\infty}\cdots\check{\integral{x_{i}=0}{\infty}}\cdots\integral{x_{n}=0}{\infty}e^{-s_{1}x_{1}\cdots-\check{s_{i}x_{i}}\cdots-s_{n}x_{n}}\left[\integral{x_{i}=0}{\infty}e^{-s_{i}x_{i}}\frac{\partial f\left(\mathbf{x}\right)}{\partial x_{i}}dx_{i}\right]dx_{1}\cdots\check{dx_{i}}\cdots dx_{n}\nonumber \\
 & =\integral{x_{1}=0}{\infty}\cdots\check{\integral{x_{i}=0}{\infty}}\cdots\integral{x_{n}=0}{\infty}e^{-s_{1}x_{1}\cdots-\check{s_{i}x_{i}}\cdots-s_{n}x_{n}}\left[\left.e^{-s_{i}x_{i}}f\left(\mathbf{x}\right)\right|_{\left(x_{1},...,x_{i-1},x_{i-1},x_{i+1},...,x_{n}\right)}^{\left(x_{1},...,x_{i-1},x_{i+1},x_{i+1},...,x_{n}\right)}\right.\nonumber \\
 & \phantom{=}\left.+s_{i}\integral{x_{i}=0}{\infty}e^{-s_{i}x_{i}}f\left(\mathbf{x}\right)dx_{i}\right]dx_{1}\cdots\check{dx_{i}}\cdots dx_{n}\nonumber \\
 & =\integral{x_{1}=0}{\infty}\cdots\check{\integral{x_{i}=0}{\infty}}\cdots\integral{x_{n}=0}{\infty}e^{-s_{1}x_{1}\cdots-\check{s_{i}x_{i}}\cdots-s_{n}x_{n}}\left[e^{-s_{i}x_{i+1}}f\left(x_{1}\ldots,x_{i-1},\underset{\text{i'th coordinate}}{\underbrace{x_{i+1}}},x_{i+1}...,x_{n}\right)\right.\nonumber \\
 & \phantom{=}\left.-e^{-s_{i}x_{i-1}}f\left(x_{1}\ldots,x_{i-1},\underset{\text{i'th coordinate}}{\underbrace{x_{i-1}}},x_{i+1}...,x_{n}\right)+s_{i}\integral{x_{i}=0}{\infty}e^{-s_{i}x_{i}}f\left(x_{1},x_{2},...,x_{n}\right)dx_{i}\right]dx_{1}\cdots\check{dx_{i}}\cdots dx_{n}\nonumber \\
 & =\integral{x_{i+1}=0}{\infty}e^{-\left(s_{i}+s_{i+1}\right)x_{i+1}}\hat{f}\left(s_{1}\ldots,s_{i-1},\underset{\text{i'th coordinate}}{\underbrace{x_{i+1}}},x_{i+1},s_{i+2}...,s_{n}\right)dx_{i+1}\nonumber \\
 & \phantom{=}-\integral{x_{i-1}=0}{\infty}e^{-\left(s_{i}+s_{i-1}\right)x_{i-1}}\hat{f}\left(s_{1}\ldots,x_{i-1},\underset{\text{i'th coordinate}}{\underbrace{x_{i-1}}},s_{i+1},s_{i+2}...,s_{n}\right)dx_{i-1}+s_{i}\hat{f}\left(s_{1},s_{2},...,s_{n}\right)\nonumber 
\end{alignat}
Note that since $\underset{x_{n}\rightarrow\infty}{\lim}e^{-x_{n}}f\left(x_{1},\ldots,x_{n}\right)=0$,
summing over the variable $i$ the first two terms in the last equality
in (\ref{eq:LT of multivariable function drivative}) cancel out for
every $i\neq1$. For $i=1$ only the second term in the brackets cancels
out and the result follows. \end{proof}
\begin{lem}
\label{lem:LT of f_k}The LT of $\mathfrak{\mathbb{D}}_{\mathbf{x}}f_{\mathbf{k}}\left(\mathbf{x}\right)$
is 
\[
\widehat{\mathfrak{\mathbb{D}}_{\mathbf{x}}f_{\mathbf{k}}}\left(\mathbf{s}\right)=\left(\sum_{i=1}^{n}s_{i}\right)\widehat{f}_{\mathbf{k}}\left(\mathbf{s}\right)-\underset{x_{1}\rightarrow0^{+}}{\lim}\hat{f}_{\mathbf{k}}\left(x_{1},s_{2},...,s_{n}\right)
\]
\end{lem}
\begin{proof}
Taking the LT of $f_{\mathbf{k}}\left(dx_{1},...,dx_{n}\right)$ first
w.r.t the indices that are not in $\mathbf{k}$ we see that 
\begin{equation}
\widehat{\mathfrak{\mathbb{D}}_{\mathbf{x}}f_{\mathbf{k}}}\left(\mathbf{x_{\mathbf{k}}},\mathbf{s}_{\mathbf{k}^{c}}\right)=\integral{\mathbb{R}_{+}^{n-l}}{}e^{-\sum_{i\in\mathbf{k}^{c}}s_{i}x_{i}}\mathbb{D}_{\mathbf{x}}f_{\mathbf{k}}\left(\mathbf{x}\right),\label{eq:f_k LT}
\end{equation}
 and can be written as 
\[
\widehat{\mathfrak{\mathbb{D}}_{\mathbf{x}}f_{\mathbf{k}}}\left(\mathbf{x_{\mathbf{k}}},\mathbf{s}_{\mathbf{k}^{c}}\right)=\mathfrak{\mathbb{D}}_{\mathbf{x}}f\left(\mathbf{x}_{\mathbf{k}}\right)e^{-\sum_{i=1}^{l-1}\left(\sum_{j=k_{i}+1}^{k_{i+1}-1}s_{j}\right)x_{k_{i}}-\sum_{j=k_{l}+1}^{n}s_{j}x_{k_{l}}},
\]
 where $f$ is an a.c function. It follows that
\begin{alignat*}{1}
\widehat{\mathfrak{\mathbb{D}}_{\mathbf{x}}f_{\mathbf{k}}}\left(\mathbf{s}\right) & =\integral{\mathbb{R}_{+}^{l}}{}e^{-\sum_{i\in\mathbf{k}}s_{i}x_{i}}\widehat{\mathfrak{\mathbb{D}}_{\mathbf{x}}f_{\mathbf{k}}}\left(\mathbf{x_{\mathbf{k}}},\mathbf{s}_{\mathbf{k}^{c}}\right)d\mathbf{x_{k}}\\
 & =\integral{\mathbb{R}_{+}^{l}}{}e^{-\sum_{i\in\mathbf{k}}s_{i}x_{i}}\mathfrak{\mathbb{D}}_{\mathbf{x}}f\left(\mathbf{x}_{\mathbf{k}}\right)e^{-\sum_{i=1}^{l-1}\left(\sum_{j=k_{i}+1}^{k_{i+1}-1}s_{j}\right)x_{k_{i}}-\sum_{j=k_{l}+1}^{n}s_{j}x_{k_{l}}}d\mathbf{x}_{\mathbf{k}}\\
 & =\integral{\mathbb{R}_{+}^{l}}{}e^{-\sum_{i=1}^{l-1}\left(\sum_{j=k_{i}}^{k_{i+1}-1}s_{j}\right)x_{k_{i}}-\sum_{j=k_{l}}^{n}s_{j}x_{k_{l}}}\mathbb{D}_{\mathbf{x}}f\left(\mathbf{x}_{\mathbf{k}}\right)d\mathbf{x}_{\mathbf{k}}.
\end{alignat*}
By Lemma \ref{lem:LT of the derivative of multivariable function}
for $n=l$ we see that
\begin{alignat*}{1}
\widehat{\mathfrak{\mathbb{D}}_{\mathbf{x}}f_{\mathbf{k}}}\left(\mathbf{s}\right) & =\left(\sum_{i=1}^{n}s_{i}\right)\integral{\mathbb{R}_{+}^{l}}{}e^{-\sum_{i=1}^{l-1}\left(\sum_{j=k_{i}}^{k_{i+1}-1}s_{j}\right)x_{k_{i}}-\sum_{j=k_{l}}^{n}s_{j}x_{k_{l}}}f\left(\mathbf{x}_{\mathbf{k}}\right)d\mathbf{x}_{\mathbf{k}}\\
 & -\lim_{x_{1}\rightarrow0^{+}}\integral{\mathbb{R}_{+}^{l-1}}{}e^{-\sum_{i=2}^{l-1}\left(\sum_{j=k_{i}}^{k_{i+1}-1}s_{j}\right)x_{k_{i}}-\sum_{j=k_{l}}^{n}s_{j}x_{k_{l}}}f\left(\mathbf{x}_{\mathbf{k}}\right)d\mathbf{x}_{\mathbf{k}'}
\end{alignat*}
where $\mathbf{k}'$ is just the vector of length $l-1$ s.t $k'_{i}=k_{i+1}$
for $1\leq i\leq l-1$. Since 
\begin{alignat*}{1}
\widehat{f_{\mathbf{k}}}\left(\mathbf{s}\right) & =\integral{\mathbb{R}_{+}^{n}}{}e^{-\sum_{i=1}^{n}s_{i}x_{i}}f\left(\mathbf{x}_{\mathbf{k}}\right)\delta_{\mathbf{k}^{c}-1}\left(d\mathbf{x}_{\mathbf{k}^{c}}\right)d\mathbf{x_{\mathbf{k}}}\\
 & =\integral{\mathbb{R}_{+}^{l}}{}e^{-\sum_{i=1}^{l-1}\left(\sum_{j=k_{i}}^{k_{i+1}-1}s_{i}\right)x_{k_{i}}-\sum_{j=k_{l}+1}^{n}s_{j}x_{k_{l}}}f\left(\mathbf{x}_{\mathbf{k}}\right)d\mathbf{x}_{\mathbf{k}},
\end{alignat*}
the result follows.
\end{proof}
If $f\left(\mathbf{x}\right)$ is a differentiable function then $\mathfrak{\mathbb{D}}_{\mathbf{x}}$
is just the directional derivative along the vector $v=(1,...,1)$
of size $n$. Let $\Psi_{x}$ be a PDO on $\mathbb{R}$ with symbol
$\psi(k)$. Then $\psi(\sum_{i=1}^{n}k_{i})$ is a symbol of the PDO
$\Psi_{\mathbf{x}}$ to be defined later and where we use bold $\mathbf{x}$
subscript to emphasize the fact that $\Psi_{\mathbf{x}}$ is defined
on functions on $\mathbb{R}^{n}$. One can think of $\Psi_{\mathbf{x}}$
as the directional version of $\Psi_{x}$ with directional vector
$v=(1,...,1)$, this will be defined rigorously in Section \ref{sec:Pseudo-Differential-Operators-on}. 

Define the RL fractional derivative of index $0<\alpha<1$ of $f\left(\mathbf{x}\right)$
to be
\begin{equation}
\mathfrak{D}_{\mathbf{x}}^{\alpha}f\left(\mathbf{x}\right)=\left(\sum_{i=1}^{n}\frac{\partial}{\partial x_{i}}\right)\integral 0{x_{1}}f\left(x_{1}-r,x_{2}-r,...,x_{n}-r\right)\frac{r^{-\alpha}}{\Gamma\left(1-\alpha\right)}dr.\label{eq:multiparameter directional fractional derivative}
\end{equation}
Once again, Equation (\ref{eq:multiparameter directional fractional derivative})
can be thought of as a fractional directional derivative. \\
As opposed to the one dimensional case where under certain conditions
the derivative w.r.t the time variable is defined on a function $p\left(x;t\right)$,
in the finite dimensional case one can not avoid the fact that $p\left(d\mathbf{x};\mathbf{t}\right)$
is $f_{\mathbf{k}}\left(d\mathbf{x}\right)$ valued on $A_{\mathbf{k}}^{n}$.
In order to describe the dynamics of $p\left(d\mathbf{x};\mathbf{t}\right)$
on $A_{\mathbf{k}}^{n}$ one should extend this notion to the functions
$f_{\mathbf{k}}\left(\mathbf{t}\right)$, where we now use the letter
\textbf{$\mathbf{t}$} in order to emphasize the context of this operator.
Since on $A_{\mathbf{k}}^{n}$ the dynamics on $\mathbf{t}_{\mathbf{k}^{c}}$
are degenerate it is reasonable to apply $\mathbb{D}_{\mathbf{x}}^{\alpha}$
on $\mathbf{t}_{\mathbf{k}}$. More precisely, if $f_{\mathbf{k}}\left(d\mathbf{t}\right)=f\left(\mathbf{t}_{\mathbf{k}}\right)\delta_{\mathbf{t}_{\mathbf{k}}-1}\left(d\mathbf{t}_{\mathbf{k}^{c}}\right)d\mathbf{t}_{\mathbf{k}}$
(here $f\left(\mathbf{t}_{\mathbf{k}}\right)$ need not be a.c) then
we define
\[
\mathfrak{D}_{\mathbf{t}}^{\alpha}f_{\mathbf{k}}\left(d\mathbf{t}\right)\coloneqq\mathbb{\mathbb{D}}_{\mathbf{t}}\left[\integral 0{x_{1}}f\left(x_{k_{1}}-r,...,x_{k_{l}}-r\right)\frac{r^{-\alpha}}{\Gamma\left(1-\alpha\right)}dr\delta_{\mathbf{t}_{\mathbf{k}}-1}\left(d\mathbf{t}_{\mathbf{k}^{c}}\right)\right].
\]
The analogue of Lemma \ref{lem:LT of f_k} is the following.
\begin{lem}
\label{lem: LT of f_k for t}The LT of \textbf{$\mathbb{\mathfrak{D}}_{\mathbf{t}}^{\alpha}f_{\mathbf{k}}\left(d\mathbf{t}\right)$
}is \textup{$\left(\sum_{i=1}^{n}s_{n}\right)^{\alpha}\widehat{f}_{\mathbf{k}}\left(\mathbf{s}\right)$.}\end{lem}
\begin{proof}
As before, we start with $f_{\mathbf{k}}\left(d\mathbf{t}\right)$
where $l=n$ so that $f_{\mathbf{k}}\left(d\mathbf{t}\right)=f\left(\mathbf{t}\right)$.
A simple computation shows that 
\begin{alignat*}{1}
\mathscr{L}\left(\integral 0{t_{1}}f\left(t_{1}-r,t_{2}-r,...,t_{n}-r\right)\frac{r^{-\alpha}}{\Gamma\left(1-\alpha\right)}dr\right)\left(\mathbf{s}\right) & =\left(\sum_{i=1}^{n}s_{i}\right)^{\alpha-1}\hat{f}\left(s_{1},\ldots s_{n}\right).
\end{alignat*}
Next, note that 
\begin{alignat*}{1}
\mathscr{L}\left(\integral 0{t_{1}}f\left(t_{1}-r,t_{2}-r,...,t_{n}-r\right)\frac{r^{-\alpha}}{\Gamma\left(1-\alpha\right)}dr\right)\left(t_{1},s_{2},...,s_{n}\right)\\
=\integral{r=0}{t_{1}}e^{-\left(\sum_{i=2}^{n}s_{i}\right)r}\widehat{f}\left(t_{1}-r,s_{2},...,s_{n}\right)dr
\end{alignat*}
so that $\underset{t_{1}\rightarrow0^{+}}{\lim}\mathscr{L}\left(\integral 0{t_{1}}f\left(t_{1}-r,t_{2}-r,...,t_{n}-r\right)\frac{r^{-\alpha}}{\Gamma\left(1-\alpha\right)}dr\right)\left(t_{1},s_{2},...,s_{n}\right)=0$.
It follows by Lemma \ref{lem:LT of the derivative of multivariable function}
that 
\[
\widehat{\mathbb{\mathfrak{D}}_{\mathbf{t}}^{\alpha}f_{\mathbf{k}}}\left(d\mathbf{t}\right)=\left(\sum_{i=1}^{n}s_{n}\right)^{\alpha}\widehat{f}\left(s_{1},\ldots s_{n}\right).
\]
The case where $l<n$ is similar to Lemma \ref{lem:LT of f_k}.\end{proof}
\begin{rem}
There is nothing exceptional about the operator $\mathbb{\mathfrak{D}}_{\mathbf{t}}^{\alpha}$,
in fact it is better to think of it as an archetype of PDOs corresponding
to Laplace symbols of L\'{e}vy measures on $\mathbb{R}_{+}$. Indeed,
if $\phi\left(s\right)=\integral{\mathbb{R}_{+}}{}\left(e^{-sy}-1\right)K_{2}\left(y\right)dy$
, then $\phi\left(s\right)$ is the symbol of the PDO $\Phi_{t}\left(f\right)\left(t\right)=\integral 0{\infty}\left(f\left(t-y\right)-f\left(t\right)\right)K_{2}\left(y\right)dy$.
A simple calculation then shows that $\phi\left(\sum_{i=1}^{n}s_{i}\right)$
is the symbol of $\Phi_{\mathbf{t}}\left(f\right)\left(\mathbf{t}\right)=\integral 0{\infty}\left(f\left(t_{1}-y,...,t_{n}-y\right)-f\left(t\right)\right)K_{2}\left(y\right)dy$
. The extension to the functions $f_{\mathbf{k}}$ is obtained along
similar lines to Lemma $\ref{lem: LT of f_k for t}$.
\end{rem}

\subsection{\label{sub:The-Semi-Markov-Approach}The Semi-Markov Approach}

Since the process $X_{t}=A_{E_{t}}$ is not Markovian, knowing its
one dimensional distribution in not enough to construct its FDDs.
To circumvent this problem Meerschaert and Straka (\cite{Meerschaert})
constructed the Markov process $\left(X_{t},R_{t}\right)$, where
$R_{t}=D_{E_{t}}-t$ is the time left before the next regeneration
of the process $X_{t}$. Let $Q_{t}\left(x',r';dx,dr\right)$ be the
transition probability of the process $\left(X_{t},R_{t}\right)$
and $0<t_{1}<t_{2}<...<t_{n}$ for some $n\in\mathbb{N}$. Then
\begin{alignat}{1}
 & P\left(X_{t_{1}}\in dx_{1},X_{t_{2}}\in dx_{2},...,X_{t_{n}}\in dx_{n}\right)\label{eq:Finite dimensional distribution general}\\
 & =\integral{r_{1}=0}{\infty}\integral{r_{2}=0}{\infty}\cdots\integral{r_{n}=0}{\infty}Q_{t_{1}}\left(0,0;dx_{1},dr_{1}\right)\nonumber \\
\times & Q_{t_{2}-t_{1}}\left(x_{1},r_{1};dx_{2},dr_{2}\right)\cdots\times Q_{t_{n}-t_{n-1}}\left(x_{n-1},r_{n-1};dx_{n},dr_{n}\right)\nonumber \\
 & =Q_{t_{1}}\left(0,0;dx_{1},dr_{1}\right)\circ Q_{t_{2}-t_{1}}\left(x_{1},r_{1};dx_{2},dr_{2}\right)\cdots Q_{t_{n}-t_{n-1}}\left(x_{n-1},r_{n-1};dx_{n},dr_{n}\right)\circ.\nonumber 
\end{alignat}
 Here, $Q_{t}\left(x',r';dx,dr\right)\circ f\left(x,r\right)=\integral{r=0}{\infty}f\left(x,r\right)Q_{t}\left(x',r';dx,dr\right)$
and $Q_{t}\left(x',r';dx,dr\right)\circ=\integral{r=0}{\infty}Q_{t}\left(x',r';dx,dr\right)$.
In \cite{Meerschaert}, the expression for $Q_{t}$ is given for a
large class of jump diffusions. Here, however, unless stated otherwise
we consider processes of the form $X_{t}=A_{E_{t}}$, where $A_{t}$
is a L\'{e}vy process and $E_{t}$ is the inverse of a strictly increasing
subordinator $D_{t}$ that is independent of $A_{t}$. That is,
\[
E_{t}=\inf\left\{ s>0:D_{s}>t\right\} .
\]
More precisely, the characteristic function of $A_{t}$ and the Laplace
transform of $D_{t}$ are given respectively by 
\begin{alignat}{1}
E\left(e^{ikA_{t}}\right) & =\exp\left[t\left(ibk-\frac{1}{2}ak^{2}+\integral{\mathbb{R}}{}\left(e^{iky}-1-iky1_{\{\left|y\right|<1\}}\right)K_{1}\left(dy\right)\right)\right]\label{eq:Charcteristic functions of A,D}\\
E\left(e^{-sD_{t}}\right) & =\exp\left[t\left(\integral{\mathbb{R}_{+}}{}\left(e^{-sy}-1\right)K_{2}\left(dy\right)\right)\right].\nonumber 
\end{alignat}
 Here, $a\geq0,b\in\mathbb{R}$. $K_{1}$ is a L\'{e}vy measure while
$K_{2}$ is a measure whose support is $[0,\infty)$ and satisfies
$\int\left(y\wedge1\right)K_{2}\left(dy\right)<\infty$, $K_{2}\left(\left\{ 0\right\} \right)=0$
and $\int K_{2}\left(dy\right)=\infty$. By (\ref{eq:Charcteristic functions of A,D})
it can be easily verified that the infinitesimal generator $\mathcal{A}$
of the process $\left(A_{t},D_{t}\right)$ is
\begin{alignat}{1}
\mathcal{A}\left(f\right)\left(x,t\right) & =b\frac{\partial}{\partial x}f\left(x,t\right)+\frac{a}{2}\frac{\partial^{2}}{\partial x^{2}}f\left(x,t\right)\label{eq:infintisimal generator of A,D}\\
 & +\integral{\mathbb{R}^{2}}{}\left(f\left(x+y,t+w\right)-f\left(x,t\right)-y\frac{\partial f(x,t)}{\partial x}1_{\{\left|\left(y,w\right)\right|<1\}}\right)K\left(dy,dw\right),\nonumber 
\end{alignat}
where $K$ is again a L\'{e}vy measure. In \cite{Meerschaert}, the case
where the coefficients $b$ and $a$ as well as the measure $K$ may
be dependent on $\left(x,t\right)$ is considered. However, when they
do not(this is referred to as the homogeneous case), the transition
probability $Q_{t}$ is given by (\cite[Equation. 4.4]{Meerschaert})
\begin{alignat}{1}
Q_{t}\left(x',r';dx,dr\right) & =1_{\{0<t<r'\}}\delta_{0}\left(dx-x'\right)\delta_{r'-t}\left(dr\right)+1_{\{0\leq r'\leq t\}}Q_{t-r'}\left(x',0;dx,dr\right)\nonumber \\
Q_{t}\left(x',0;dx,dr\right) & =\integral{y\in\mathbb{R}}{}\integral{w\in[0,t]}{}U^{x'}\left(dy,dw\right)K\left(dx-y,dr+t-w\right),\label{eq:general transition probabilities}
\end{alignat}
where $U^{x'}\left(dy,dw\right)$ is the occupation measure of $\left(A_{t},D_{t}\right)$,
i.e
\begin{alignat*}{1}
\intop f\left(y,w\right)U^{x'}\left(dy,dw\right) & =\mathbb{E}\left(\integral 0{\infty}f\left(A_{u}+x',D_{u}\right)du\right).
\end{alignat*}
When the processes $A_{t}$ and $D_{t}$ are independent, it can be
easily verified that 
\begin{equation}
U^{x'}\left(dy,dw\right)=\integral 0{\infty}z\left(dy-x',u\right)g\left(dw,u\right)du,\label{eq:Occupation measure of UCTRWL}
\end{equation}
where $z\left(dx,t\right)=P\left(A_{t}\in dx\right)$ and $g\left(dx,t\right)=P\left(D_{t}\in dx\right)$.
Moreover, in the case of independence it was shown that (\cite[Corollary 2.3]{becker2004limit})
\[
K\left(dy,dw\right)=K_{1}\left(dy\right)\delta_{0}\left(dw\right)+\delta_{0}\left(dy\right)K_{2}\left(dw\right).
\]
Hence, equations (\ref{eq:general transition probabilities}) translate
into

\begin{alignat}{1}
Q_{t}\left(x',r';dx,dr\right) & =1_{\{0<t<r'\}}\delta_{0}\left(dx-x'\right)\delta_{r'-t}\left(dr\right)\nonumber \\
 & +1_{\{0\leq r'\leq t\}}\integral{y\in\mathbb{R}}{}\integral{\,w\in[0,t-r']}{}\left(\integral 0{\infty}z\left(dy-x',u\right)g\left(dw,u\right)du\right)\nonumber \\
 & \times\left(\delta_{0}\left(dr+t-r'-w\right)K_{1}\left(dx-y\right)+\delta_{0}\left(dx-y\right)K_{2}\left(dr+t-r'-w\right)\right).\label{eq:Transition probability of UCTRWL}
\end{alignat}
However, since $\int K_{2}\left(dy\right)=\infty$, we see (\cite[Theorem. 27.4]{sato1999L\'{e}vy})
that $g\left(dw,t\right)$ has no atoms. Therefore, (\ref{eq:Transition probability of UCTRWL})
reduces to
\begin{alignat}{1}
Q_{t}\left(x',r';dx,dr\right) & =1_{\{0\leq t<r'\}}\delta_{0}\left(dx-x'\right)\delta_{r'-t}\left(dr\right)\nonumber \\
 & +1_{\{0\leq r'\leq t\}}\integral{\,w\in[0,t-r']}{}\left(\integral 0{\infty}z\left(dx-x',u\right)g\left(dw,u\right)du\right)\label{eq:transition probaiblity of UCTRWL2}\\
 & \times K_{2}\left(dr+t-r'-w\right).\nonumber 
\end{alignat}

\section{\label{sec:Fokker-Planck-Equations}Fokker-Planck Equations}

Throughout this section, we let $A_{t}$ be a L\'{e}vy process such that
$E\left(e^{ikA_{t}}\right)=e^{t\psi\left(k\right)}$, its probability
density is given by $z\left(dx,t\right)=P\left(A_{t}\in dx\right)$.
$E_{t}$ is the inverse of a subordinator $D_{t}$ such that $E\left(e^{-sD_{t}}\right)=e^{t\phi\left(s\right)}$,
its probability density is $h\left(dx,t\right)=P\left(E_{t}\in dx\right)$.
We denote by $\Psi$ and $\Phi$ the pseudo-differential operators
of the symbols $\psi\left(-k\right)$ and $-\phi\left(s\right)$ respectively.
We also denote the transition probability function of the Markov process
$\left(X_{t},R_{t}\right)$ by $Q_{t}$ and that of $\left(E_{t},R_{t}\right)$
by $H_{t}$. Next note that the occupation measure of $\left(t,E_{t}\right)$
is just $U^{x'}\left(dx,dw\right)=g\left(dw,x-x'\right)dx$(cf. \cite[Eq. 5.1]{Meerschaert}),
and similarly to (\ref{eq:transition probaiblity of UCTRWL2}) we
have
\begin{alignat}{1}
H_{t}\left(x',r';dx,dr\right) & =1_{\{0\leq t<r'\}}\delta_{0}\left(dx-x'\right)\delta_{r'-t}\left(dr\right)\label{eq:transition of H}\\
 & +1_{\{0\leq r'\leq t\}}\integral{\,w\in[0,t-r']}{}g\left(dw,x-x'\right)dx\times K_{2}\left(dr+t-r'-w\right).\nonumber 
\end{alignat}
The next theorem finds the FFPE of the FDD of $E_{t}$.

\begin{thm}
\label{thm:governing equation of multi variable h} Let $h\left(dx_{1},...,dx_{n};t_{1},...,t_{n}\right)$
be the FDD of $E_{t}$ where $t_{1}<t_{2}<...<t_{n}$, i.e 
\[
h\left(dx_{1},...,dx_{n};t_{1},...,t_{n}\right)=P\left(E_{t_{1}}\in dx_{1},...,E_{t_{n}}\in dx_{n}\right).
\]
 Then 
\begin{equation}
\Phi_{\mathbf{t}}h\left(d\mathbf{x};\mathbf{t}\right)=-\mathbb{D}_{\mathbf{x}}h\left(d\mathbf{x};\mathbf{t}\right).\label{eq:governing equation of multi variable h}
\end{equation}
\end{thm}
\begin{proof}
Let us take LT with respect to the spatial variables and with respect
to the time variables, this will be abbreviated by LLT. Before taking
the LLT of $h\left(d\mathbf{x};\mathbf{t}\right)$ we note that since
$H_{t}\left(x',r';dx,dr\right)$ is translation invariant with respect
to the spatial variable we have
\begin{alignat}{1}
 & h\left(d\mathbf{x};\mathbf{t}\right)\nonumber \\
 & =H_{t_{1}}\left(0,0;dx_{1},dr_{1}\right)\circ H_{t_{2}-t_{1}}\left(0,r_{1};dx_{2}-x_{1},dr_{2}\right)\cdots H_{t_{n}-t_{n-1}}\left(0,r_{n-1};dx_{n}-x_{n-1},dr_{n}\right)\circ.\label{eq:FD distribution of E}
\end{alignat}
Taking the LLT of (\ref{eq:FD distribution of E}), by a simple change
of variables we see that(to avoid confusion we now use $\lambda$
instead of $k$)
\begin{alignat}{1}
 & \widehat{h}\left(\lambda_{1},\ldots\lambda_{n};s_{1},\ldots s_{n}\right)\nonumber \\
 & =\integral{t_{1}=0}{\infty}\integral{x_{1}=0}{\infty}e^{-\left(\sum_{i=1}^{n}s_{i}\right)t_{1}-\left(\sum_{i=1}^{n}\lambda_{i}\right)x_{1}}H_{t_{1}}\left(0,0;dx_{1},dr_{1}\right)\circ dt_{1}\label{eq:governing equation of h}\\
 & \widehat{H}_{\sum_{i=2}^{n}s_{i}}\left(0,r_{1};\sum_{i=2}^{n}\lambda_{i},dr_{2}\right)\circ\cdots\widehat{H}_{s_{n}+s_{n-1}}\left(0,r_{n-2};\lambda_{n}+\lambda_{n-1},dr_{n-1}\right)\circ\widehat{H}_{s_{n}}\left(0,r_{n-1};\lambda_{n},dr_{n}\right)\circ.\nonumber 
\end{alignat}
Now, let us look at
\begin{alignat}{1}
 & \integral{t_{1}=0}{\infty}\integral{x_{1}=0}{\infty}e^{-\left(\sum_{i=1}^{n}s_{i}\right)t_{1}-\left(\sum_{i=1}^{n}\lambda_{i}\right)x_{1}}H_{t_{1}}\left(0,0;dx_{1},dr_{1}\right)dt_{1}\nonumber \\
 & =\integral{t_{1}=0}{\infty}\integral{x_{1}=0}{\infty}e^{-\left(\sum_{i=1}^{n}s_{i}\right)t_{1}-\left(\sum_{i=1}^{n}\lambda_{i}\right)x_{1}}\integral{\,w\in[0,t_{1}]}{}g\left(w,x_{1}\right)dx_{1}\nonumber \\
 & \times K_{2}\left(dr_{1}+t_{1}-w\right)dw\nonumber \\
 & =\integral{x_{1}=0}{\infty}e^{-\left(\sum_{i=1}^{n}\lambda_{i}\right)x_{1}}\integral{\,w\in[0,\infty]}{}g\left(w,x_{1}\right)dx_{1}\nonumber \\
 & \times\integral{t_{1}=w}{\infty}e^{-\left(\sum_{i=1}^{n}s_{i}\right)t_{1}}K_{2}\left(dr_{1}+t_{1}-w\right)dw\nonumber \\
 & =\integral{x_{1}=0}{\infty}e^{-\left(\sum_{i=1}^{n}\lambda_{i}\right)x_{1}}\integral{\,w\in[0,\infty]}{}g\left(w,x_{1}\right)dx_{1}e^{-\left(\sum_{i=1}^{n}s_{i}\right)w}dw\nonumber \\
 & \times\integral{t_{1}=0}{\infty}e^{-\left(\sum_{i=1}^{n}s_{i}\right)t_{1}}K_{2}\left(dr_{1}+t_{1}\right)\nonumber \\
 & =\frac{1}{\sum_{i=1}^{n}\lambda_{i}-\phi\left(\sum_{i=1}^{n}s_{i}\right)}\integral{t_{1}=0}{\infty}e^{-\left(\sum_{i=1}^{n}s_{i}\right)t_{1}}K_{2}\left(dr_{1}+t_{1}\right).\label{eq:governing equation of h-1}
\end{alignat}
Next note that,
\begin{alignat}{1}
 & \underset{x_{1}\rightarrow0^{+}}{\lim}\widehat{h}\left(x_{1},\lambda_{2},\ldots\lambda_{n};s_{1},\ldots s_{n}\right)\label{eq:governing equation of h-2}\\
 & =\underset{x_{1}\rightarrow0^{+}}{\lim}\integral{t_{1}=0}{\infty}e^{-\left(\sum_{i=1}^{n}s_{i}\right)t_{1}-\left(\sum_{i=2}^{n}\lambda_{i}\right)x_{1}}\integral{\,w\in[0,t_{1}]}{}g\left(dw,x_{1}\right)\times\integral{r_{1}=0}{\infty}K_{2}\left(dr_{1}+t_{1}-w\right)\nonumber \\
 & \phantom{}\times\widehat{H}_{\sum_{i=2}^{n}s_{i}}\left(0,r_{1};\sum_{i=2}^{n}\lambda_{i},dr_{2}\right)\circ\cdots\widehat{H}_{s_{n}+s_{n-1}}\left(0,r_{n-2};\lambda_{n}+\lambda_{n-1},dr_{n-1}\right)\circ\widehat{H}_{s_{n}}\left(0,r_{n-1};\lambda_{n},dr_{n}\right)\circ.\nonumber \\
 & =\integral{t_{1}=0}{\infty}e^{-\left(\sum_{i=1}^{n}s_{i}\right)t_{1}}\integral{r_{1}=0}{\infty}K_{2}\left(dr_{1}+t_{1}\right)\nonumber \\
 & \phantom{}\times\widehat{H}_{\sum_{i=2}^{n}s_{i}}\left(0,r_{1};\sum_{i=2}^{n}\lambda_{i},dr_{2}\right)\cdots\widehat{H}_{s_{n}+s_{n-1}}\left(0,r_{n-2};\lambda_{n}+\lambda_{n-1},dr_{n-1}\right)\circ\widehat{H}_{s_{n}}\left(0,r_{n-1};\lambda_{n},dr_{n}\right)\circ.\nonumber 
\end{alignat}
 Indeed, by the continuity of the measure $K_{2}$ and \cite[Lemma 27.1]{sato1999levy}
follows the continuity of the following function 
\[
w\mapsto\integral{r_{1}=0}{\infty}K_{2}\left(dr_{1}+t_{1}-w\right)\phantom{}\times\widehat{H}_{\sum_{i=2}^{n}s_{i}}\left(0,r_{1};\sum_{i=2}^{n}\lambda_{i},dr_{2}\right)\circ\cdots\widehat{H}_{s_{n}}\left(0,r_{n-1};\lambda_{n},dr_{n}\right)\circ,
\]
since $g\left(dw,x_{1}\right)dx_{1}$ converges weakly to $\delta_{0}\left(dw\right)$
as $x_{1}\rightarrow0^{+}$ (\ref{eq:governing equation of h-2})
follows. Finally, plugging (\ref{eq:governing equation of h-1}) in
(\ref{eq:governing equation of h}), using (\ref{eq:governing equation of h-2})
and rearranging terms we arrive at 
\begin{equation}
-\phi\left(\sum_{i=1}^{n}s_{i}\right)\widehat{h}\left(\lambda_{1},\ldots\lambda_{n};s_{1},\ldots s_{n}\right)=-\left(\sum_{i=1}^{n}\lambda_{i}\right)\widehat{h}\left(\lambda_{1},\ldots\lambda_{n};s_{1},\ldots s_{n}\right)+\widehat{h}\left(0^{+},\lambda_{2},\ldots\lambda_{n};s_{1},\ldots s_{n}\right).\label{eq:governing equation of h-3}
\end{equation}
Taking the inverse LLT of (\ref{eq:governing equation of h-3}) and
using Lemma \ref{lem:LT of f_k} we obtain (\ref{eq:governing equation of multi variable h}).
\end{proof}
Theorem \ref{thm:governing equation of multi variable h} paves the
way for the finite dimensional FFPEs of the process $X_{t}$. We denote
the FDD of $A_{t}$ by $z\left(dx_{1},...,dx_{n};t_{1},...,t_{n}\right)=P\left(A_{t_{1}}\in dx_{1},...,A_{t_{n}}\in dx_{n}\right)$.
\begin{cor}
\label{cor: MD FPE of A_E_t} Let $p\left(dx_{1},...,dx_{n};t_{1},...,t_{n}\right)=P\left(X_{t_{1}}\in dx_{1},\ldots,X_{t_{n}}\in dx_{n}\right)$
where $t_{1}<t_{2}<...<t_{n}$. Then
\begin{alignat}{1}
\Phi_{\mathbf{t}}p\left(dx_{1},...,dx_{n};t_{1},...,t_{n}\right) & =\Psi_{\mathbf{x}}p\left(dx_{1},...,dx_{n};t_{1},...,t_{n}\right)\nonumber \\
 & +\integral{u_{2}=0}{\infty}\integral{u_{3}=u_{2}}{\infty}\cdots\integral{u_{n}=u_{n}-1}{\infty}\delta_{0}\left(dx_{1}\right)z\left(dx_{2},...,dx_{n};u_{2},...,u_{n}\right)h\left(0^{+},du_{2},\ldots,du_{n};t_{1},\ldots,t_{n}\right)\label{eq:multidimensional FPE of A_E_t}
\end{alignat}
\end{cor}
\begin{proof}
By the independence of $A_{t}$ and $D_{t}$ 
\begin{alignat}{1}
 & p\left(dx_{1},...,dx_{n};t_{1},...,t_{n}\right)\label{eq:multidimensional FPE of A_E_t-1}\\
 & =\integral{u_{1}=0}{\infty}\integral{u_{2}=u_{1}}{\infty}\cdots\integral{u_{n}=u_{n-1}}{\infty}z\left(dx_{1},...,dx_{n};u_{1},...,u_{n}\right)h\left(du_{1},\ldots,du_{n};t_{1},\ldots,t_{n}\right)\nonumber \\
 & =\integral{u_{1}=0}{\infty}\cdots\integral{u_{n}=u_{n-1}}{\infty}z\left(dx_{1},u_{1}\right)z\left(dx_{2}-x_{1},dx_{3}-x_{1},...,dx_{n}-x_{1};u_{2}-u_{1},u_{3}-u_{1},...,u_{n}-u_{1}\right)\nonumber \\
 & \times H_{t_{1}}\left(0,0;du_{1},dr_{1}\right)\circ H_{t_{2}-t_{1}}\left(0,r_{1};du_{2}-u_{1},dr_{2}\right)\circ\cdots H_{t_{n}-t_{n-1}}\left(0,r_{n-1};du_{n}-u_{n-1},dr_{n}\right)\circ.
\end{alignat}
Taking the FLT of $p\left(dx_{1},...,dx_{n};t_{1},...,t_{n}\right)$
and using the change of variables $u_{i}^{'}=u_{i}-u_{1}$ for $2\leq i\leq n$
we obtain 
\begin{alignat}{1}
 & \overline{p}\left(k_{1},...,k_{n};s_{1},...,s_{n}\right)\nonumber \\
 & =\integral{u_{1}=0}{\infty}\integral{t_{1}=0}{\infty}\integral{x_{1}\in\mathbb{R}}{}e^{-\left(\sum_{i=1}^{n}s_{i}\right)t_{1}-\left(i\sum_{i=1}^{n}k_{i}\right)x_{1}}z\left(dx_{1},u_{1}\right)H_{t_{1}}\left(0,0;du_{1},dr_{1}\right)\circ dt_{1}\nonumber \\
 & \times\integral{u_{2}=0}{\infty}\cdots\integral{u_{n}=u_{n-1}}{\infty}\widetilde{z}\left(k_{2},...,k_{n};u_{2},...,u_{n}\right)\hat{H}_{\sum_{i=2}^{n}s_{i}}\left(0,r_{1};du_{2},dr_{2}\right)\circ\cdots\hat{H}_{s_{n}}\left(0,r_{n-1};du_{n}-u_{n-1},dr_{n}\right)\circ\label{eq:multidimensional FPE of A_E_t-2}
\end{alignat}
Let us look at 
\begin{alignat}{1}
 & \integral{u_{1}=0}{\infty}\integral{t_{1}=0}{\infty}\integral{x_{1}\in\mathbb{R}}{}e^{-\left(\sum_{i=1}^{n}s_{i}\right)t_{1}-\left(i\sum_{i=1}^{n}k_{i}\right)x_{1}}z\left(dx_{1},u_{1}\right)H_{t_{1}}\left(0,0;du_{1},dr_{1}\right)dt_{1}\nonumber \\
 & =\integral{u_{1}=0}{\infty}\integral{t_{1}=0}{\infty}\integral{x_{1}\in\mathbb{R}}{}e^{-\left(\sum_{i=1}^{n}s_{i}\right)t_{1}-\left(i\sum_{i=1}^{n}k_{i}\right)x_{1}}z\left(dx_{1},u_{1}\right)\nonumber \\
 & \phantom{}\times\integral{\,w\in[0,t_{1}]}{}g\left(w,u_{1}\right)du_{1}K_{2}\left(dr_{1}+t_{1}-w\right)dt_{1}\nonumber \\
 & =\integral{u_{1}=0}{\infty}\integral{w=0}{\infty}\integral{x_{1}\in\mathbb{R}}{}e^{-i\left(\sum_{i=1}^{n}k_{i}\right)x_{1}-\left(\sum_{i=1}^{n}s_{i}\right)w}z\left(dx_{1},u_{1}\right)g\left(w,u_{1}\right)du_{1}\nonumber \\
 & \phantom{}\times\integral{t_{1}=0}{\infty}e^{-\left(\sum_{i=1}^{n}s_{i}\right)t_{1}}K_{2}\left(dr_{1}+t_{1}\right)dt_{1}\nonumber \\
 & =\integral{u_{1}=0}{\infty}\integral{w=0}{\infty}\integral{x_{1}\in\mathbb{R}}{}e^{u_{1}\left(\psi\left(-\sum_{i=1}^{n}k_{i}\right)+\phi\left(s\sum_{i=1}^{n}s_{i}\right)\right)}du_{1}\times\integral{t_{1}=0}{\infty}e^{-\left(\sum_{i=1}^{n}s_{i}\right)t_{1}}K_{2}\left(dr_{1}+t_{1}\right)dt_{1}\\
 & =\frac{1}{-\psi\left(-\sum_{i=1}^{n}k_{i}\right)-\phi\left(\sum_{i=1}^{n}s_{i}\right)}\integral{t_{1}=0}{\infty}e^{-\left(\sum_{i=1}^{n}s_{i}\right)t_{1}}K_{2}\left(dr_{1}+t_{1}\right)dt_{1}.\label{eq:eq:multidimensional FPE of A_E_t-3}
\end{alignat}
Plugging (\ref{eq:eq:multidimensional FPE of A_E_t-3}) in (\ref{eq:multidimensional FPE of A_E_t-2})
and using (\ref{eq:governing equation of h-2}) we have 
\begin{alignat*}{1}
\overline{p}\left(k_{1},...,k_{n};s_{1},...,s_{n}\right) & =\frac{1}{-\psi\left(-\sum_{i=1}^{n}k_{i}\right)-\phi\left(\sum_{i=1}^{n}s_{i}\right)}\\
 & \times\integral{u_{2}=0}{\infty}\cdots\integral{u_{n}=u_{n-1}}{\infty}\widetilde{z}\left(k_{2},...,k_{n};u_{2},...,u_{n}\right)\hat{h}\left(0^{+},du_{2},\ldots,du_{n};s_{1},\ldots,s_{n}\right).
\end{alignat*}
Rearranging and taking the inverse FLT we arrive at (\ref{eq:multidimensional FPE of A_E_t}).
\end{proof}
Working along similar lines to the proof of Theorem \ref{thm:governing equation of multi variable h}
one can also obtain the finite dimensional FFPEs of the process $X_{t}=A_{E_{t}}$
where $E_{t}$ is the inverse of a strictly increasing subordinator
$D_{t}$ and $\left(A_{t},D_{t}\right)$ is a L\'{e}vy process, i.e. the
processes $A_{t}$ and $D_{t}$ are not necessarily independent. More
precisely, suppose $E\left(e^{ikA_{t}-sD_{t}}\right)=e^{t\xi\left(k,s\right)}$
and that $\xi\left(k,s\right)=ibk-\frac{1}{2}ak^{2}+\integral{\mathbb{R}}{}\left(e^{iky-sw}-1-iky1_{\{\left|\left(y,w\right)\right|<1\}}\right)K\left(dy,dw\right)$
and that $\Xi$ is the operator whose symbol is $-\xi\left(-k,s\right)$.

\begin{cor}
\label{cor:MD FPE of general A_E_t} Let $\left(A_{t},D_{t}\right)$
be a L\'{e}vy process s.t $E\left(e^{ikA_{t}-sD_{t}}\right)=e^{t\xi\left(k,s\right)}$.
Let $E_{t}$ be the inverse of the strictly increasing subordinator
$D_{t}$ and let $p\left(dx_{1},...,dx_{n};t_{1},...,t_{n}\right)=P\left(X_{t_{1}}\in dx_{1},\ldots,X_{t_{n}}\in dx_{n}\right)$.
Then
\begin{alignat}{1}
\Xi_{\mathbf{x},\mathbf{t}}p\left(dx_{1},...,dx_{n};t_{1},...,t_{n}\right) & =\integral{r_{1}=0}{\infty}K\left(dx_{1},dr_{1}+t_{1}\right)\label{eq:multidimensional FPE of general A_E_t}\\
 & \phantom{}\times Q_{t_{2}-t_{1}}\left(x_{1},r_{1};dx_{2},dr_{2}\right)\circ\cdots Q_{t_{n}-t_{n-1}}\left(x_{n-1},r_{n-1};dx_{n},dr_{n}\right)\circ.\nonumber 
\end{alignat}
\end{cor}
\begin{proof}
Using (\ref{eq:general transition probabilities}) we see that $Q_{t}$
is again translation invariant with respect to the spatial variable.
Note that here
\[
U^{x'}\left(dy,dw\right)=\integral 0{\infty}v\left(dy-x',dw;u\right)du,
\]
where $v\left(dy,dw;u\right)=P\left(A_{u}\in dy,D_{u}\in dw\right)$.
Using the same ideas as in the proof of Theorem \ref{thm:governing equation of multi variable h}
we obtain
\begin{alignat}{1}
\overline{p}\left(k_{1},...,k_{n};s_{1},...,s_{n}\right) & =\integral{t_{1}=0}{\infty}\integral{x_{1}\in\mathbb{R}}{}e^{-\left(\sum_{i=1}^{n}s_{i}\right)t_{1}-\left(\sum_{i=1}^{n}k_{i}\right)x_{1}}\integral{u=0}{\infty}v\left(dy,dw;u\right)du\integral{r_{1}=0}{\infty}\integral{y\in\mathbb{R}}{}\integral{w=0}{t_{1}}K\left(dx_{1}-y,dr_{1}+t_{1}-w\right)\label{eq:multidimensional FPE of general A_E_t FLT}\\
 & \times\overline{Q}_{\sum_{i=2}^{n}s_{i}}\left(0,dr_{1};\sum_{i=2}^{n}k_{i},dr_{2}\right)\circ\cdots\overline{Q}_{s_{n}}\left(0,r_{n-1};k_{n},dr_{n}\right)\circ\nonumber \\
 & =\integral{y\in\mathbb{R}}{}\integral{w=0}{\infty}e^{-\left(\sum_{i=1}^{n}s_{i}\right)w-i\left(\sum_{i=1}^{n}k_{i}\right)y}\integral{u=0}{\infty}v\left(dy,dw;u\right)du\integral{r_{1}=0}{\infty}\integral{t_{1}=0}{\infty}\integral{x_{1}\in\mathbb{R}}{}e^{-\left(\sum_{i=1}^{n}s_{i}\right)t_{1}-\left(\sum_{i=1}^{n}k_{i}\right)x_{1}}\nonumber \\
 & \times K\left(dx_{1},dr_{1}+t_{1}\right)\overline{Q}_{\sum_{i=2}^{n}s_{i}}\left(0,dr_{1};\sum_{i=2}^{n}k_{i},dr_{2}\right)\circ\cdots\overline{Q}_{s_{n}}\left(0,r_{n-1};k_{n},dr_{n}\right)\circ\nonumber \\
 & =\frac{1}{-\xi\left(-\sum_{i=1}^{n}k_{i},\sum_{i=1}^{n}s_{i}\right)}\integral{r_{1}=0}{\infty}\integral{t_{1}=0}{\infty}\integral{x_{1}\in\mathbb{R}}{}e^{-\left(\sum_{i=1}^{n}s_{i}\right)t_{1}-\left(\sum_{i=1}^{n}k_{i}\right)x_{1}}K\left(dx_{1},dr_{1}+t_{1}\right)\nonumber \\
 & \times\overline{Q}_{\sum_{i=2}^{n}s_{i}}\left(0,dr_{1};\sum_{i=2}^{n}k_{i},dr_{2}\right)\circ\cdots\overline{Q}_{s_{n}}\left(0,r_{n-1};k_{n},dr_{n}\right)\circ.\nonumber 
\end{alignat}
Rearrange and invert to obtain (\ref{eq:multidimensional FPE of general A_E_t}).\end{proof}
\begin{rem}
When the CTRW is coupled, a distinction between the limit of the CTRW
$X'_{t}=\sum_{k=1}^{N_{t}}J_{k}$ and the Overshoot CTRW $X''_{t}=\sum_{k=1}^{N_{t}+1}J_{k}$
is needed. Indeed, in the case where the outer process $A_{t}$ and
the subordinator $D_{t}$ are dependent it has been proven in \cite{straka2011}
that the limits of the CTRW and the Overshoot CTRW are $A_{\left(E_{t}\right)-}$
and $A_{E_{t}}$ respectively.
\end{rem}

\begin{rem}
As was mentioned above, it is usually impossible to define CTRWL by
their one-dimensional FFPE. However, since Equation (\ref{eq:multidimensional FPE of general A_E_t})
and Equation (\ref{eq:multidimensional FPE of general A_E_t FLT})
are equivalent and c�dl�g processes are characterized(up to their
law) by their FDDs, we see that one can define the process $A_{E_{t}}$
by specifying all its $n$ dimensional FFPE.
\end{rem}
Our next result gives a meaning to the measure $\integral{r_{1}=0}{\infty}K\left(dx_{1},dr_{1}+t_{1}\right)$
in the context of CTRWL.
\begin{prop}
\label{prop:Remainder results}Let $A_{t}$ and $D_{t}$ as in Corollary
\ref{cor:MD FPE of general A_E_t}. Then
\begin{equation}
\frac{\partial}{\partial u}P\left(X_{t}\in dx,E_{t}\leq u\right)\overset{w}{\rightarrow}\integral{r_{1}=0}{\infty}K\left(dx_{1},dr_{1}+t_{1}\right)\qquad u\rightarrow0.\label{eq:Remainder results}
\end{equation}
\end{prop}
\begin{proof}
Let $A'_{t}=\left(A_{t},t\right)$ and note that $A'_{E_{t}}=\left(X_{t},E_{t}\right)$
. Using \cite[Equation 4.4]{Meerschaert} (which is Equation (\ref{eq:general transition probabilities})
for outer process in $\mathbb{R}^{d}$) for every $x_{1}\in\mathbb{R}$
we have
\begin{alignat}{1}
P\left(X_{t}\in(-\infty,x_{1}],E_{t}\leq q\right) & =\integral{y_{1}\in\mathbb{R}}{}\integral{y_{2}\in\mathbb{R}}{}\integral{w=0}t\integral{u=0}{\infty}\integral{r=0}{\infty}v'\left(dy_{1},dy_{2},dw;u\right)duK'\left((-\infty,x_{1}-y_{1}],dx_{2}-y_{2},r+t-w\right).\label{eq:Remainder results 1}
\end{alignat}
It is not hard to see that here 
\begin{alignat}{1}
v'\left(dy_{1},dy_{2},dw;u\right) & =v\left(dy_{1},dw;u\right)\delta_{u}\left(dy_{2}\right)\label{eq:Remainder results 2}\\
K'\left(dx_{1},dx_{2},dw\right) & =K\left(dx_{1},dw\right)\delta_{0}\left(dx_{2}\right).\nonumber 
\end{alignat}
Plugging (\ref{eq:Remainder results 2}) in (\ref{eq:Remainder results 1})
we have
\begin{alignat}{1}
 & \integral{y_{1}\in\mathbb{R}}{}\integral{y_{2}\in\mathbb{R}}{}\integral{w=0}t\integral{u=0}{\infty}v\left(dy_{1},dw;u\right)\delta_{u}\left(dy_{2}\right)duK\left((-\infty,x_{1}-y_{1}],[t-w,\infty)\right)\delta_{0}\left(dx_{2}-y_{2}\right)\nonumber \\
 & =\integral{y_{1}\in\mathbb{R}}{}\integral{w=0}t\integral{u=0}{\infty}v\left(dy_{1},dw;u\right)duK\left((-\infty,x_{1}-y_{1}],[t-w,\infty)\right)\delta_{u}\left(dx_{2}\right).\label{eq:Remainder results3}
\end{alignat}
Integrating w.r.t $x_{2}$ on $[0,q]$ for some $q>0$ we have
\begin{alignat}{1}
P\left(X_{t}\in(-\infty,x_{1}],E_{t}\leq q\right) & =\integral{y_{1}\in\mathbb{R}}{}\integral{w=0}t\integral{u=0}{\infty}v\left(dy_{1},dw;u\right)duK\left((-\infty,x_{1}-y_{1}],[t-w,\infty)\right)1_{\left\{ u\leq q\right\} }\nonumber \\
 & =\integral{y_{1}\in\mathbb{R}}{}\integral{w=0}t\integral{u=0}qv\left(dy_{1},dw;u\right)duK\left((-\infty,x_{1}-y_{1}],[t-w,\infty)\right).\label{eq:Remainder results4}
\end{alignat}
Taking derivative w.r.t $q$ we have
\[
\frac{\partial}{\partial q}P\left(X_{t}\in(-\infty,x_{1}],E_{t}\leq q\right)=\integral{y_{1}\in\mathbb{R}}{}\integral{w=0}tv\left(dy_{1},dw;q\right)duK\left((-\infty,x_{1}-y_{1}],[t-w,\infty)\right).
\]
The measure $K\left(dx_{1},dw\right)$ is continuous because $K\left(\mathbb{R},dw\right)=K_{2}\left(dw\right)$
is continuous. Letting $q\rightarrow0$ we see that $v\left(dy_{1},dw;q\right)\overset{w}{\rightarrow}\delta_{\left(0,0\right)}\left(dy_{1},dw\right)$
and hence $\frac{\partial}{\partial q}P\left(X_{t}\in(-\infty,x_{1}],E_{t}\leq q\right)\rightarrow K\left((-\infty,x_{1}],[t,\infty)\right)$
as $q\rightarrow0$ which is equivalent to (\ref{eq:Remainder results}).
\end{proof}
Proposition \ref{prop:Remainder results} sheds light on the ``remainder''
term that appears in Corollary \ref{cor: MD FPE of A_E_t} and Corollary
\ref{cor:MD FPE of general A_E_t}. It appears that using a so-called
multidimensional RL PDO in the finite dimensional FFPE one is left
with a term that accounts for the portion of particles that have not
been mobilized since $t=0$. More philosophically, if we think of
the value of $E_{t}$ as the number of mobilizations of the process
by time $t$, then $\frac{\partial}{\partial u}P\left(E_{t}\leq u\right)$
is the ratio between the portion of particles $\Delta m$ that experienced
between $u$ and $u+\mbox{\ensuremath{\Delta}}u$ mobilizations up
to time $t$. Evaluating $\frac{\partial}{\partial u}P\left(E_{t}\leq u\right)$
at $u=0$ is then the ratio between the portion of particles $\Delta m$
that experienced an infinitesimal number of mobilizations $\Delta u$
by time $t$ and $\Delta u$. If $\frac{\partial}{\partial u}P\left(E_{t}\leq u\right)|_{u=0}$
is big then the diffusion becomes very dynamic at time $t$ as many
particles get loose and ``take part'' in the diffusion. Considering
now $\frac{\partial}{\partial u}P\left(X_{t}\in dx,E_{t}\leq u\right)|_{u=0}$
we see that since $X_{t}$ is the limit of the Overshooting CTRW,
where a jump precedes a waiting time the position of the particle
that has been ``stuck'' until time $t$ depends on that first jump
in space. It that context is worthwhile to compare this to \cite[Equation 4.2]{Jurlewicz2012},
the dynamics of the coupled CTRWL where the jump in space succeeds
that in time. The ``remainder'' term therefore accounts for the
Finite dimensional dynamics that only ``kick in'' at time $t$.

\section{\label{sec:Examples}Examples}

Theorem \ref{thm:governing equation of multi variable h} as well
as Corollary \ref{cor: MD FPE of A_E_t} and Corollary \ref{cor:MD FPE of general A_E_t}
should be compared with their one-dimensional counterparts to gain
a better understanding of the dynamics of the processes whose distributions
govern the FFPE. We start with a specific case of the one dimensional
analogue of Theorem \ref{thm:governing equation of multi variable h}.
\begin{example}
Let $D_{t}$ be a standard stable subordinator of index $0<\alpha<1$,
i.e. $E\left(e^{-sD_{t}}\right)=e^{t\left(-s^{\alpha}\right)}$. Its
inverse $E_{t}$ has a distribution $h\left(x,t\right)$ which satisfies
( \cite[Equation 5.5]{Meerschaert2013}) 
\[
\mathfrak{D}_{t}^{\alpha}h\left(x,t\right)=-\mathbb{D}_{x}h\left(x,t\right),
\]
on $x,t>0$. Since here $\phi\left(s\right)=-s^{\alpha}$, we see
that $\Phi_{t}=\mathfrak{D}_{t}^{\alpha}$.
\end{example}
Next we look at the one dimensional analogue of Corollary \ref{cor: MD FPE of A_E_t}. 
\begin{example}
Again we let $D_{t}$ be a standard stable subordinator of index $0<\alpha<1$
, and $A_{t}$ be a L\'{e}vy process s.t $E\left(e^{ikA_{t}}\right)=e^{t\psi\left(k\right)}$.
Then the distribution $p\left(dx,t\right)$ of $A_{E_{t}}$ satisfies
( \cite[Equation 5.6]{Meerschaert2013})
\begin{equation}
\mathfrak{D}_{t}^{\alpha}p\left(dx,t\right)=\Psi_{x}p\left(dx,t\right)+\frac{t^{-\alpha}}{\Gamma\left(1-\alpha\right)}\delta_{0}\left(dx\right).\label{eq:one dimensional FPE of A_E_t}
\end{equation}
To see why (\ref{eq:multidimensional FPE of A_E_t}) can be thought
of as a generalization of (\ref{eq:one dimensional FPE of A_E_t})
note that $h\left(0^{+},t\right)=\frac{t^{-\alpha}}{\Gamma\left(1-\alpha\right)}$
(\cite[Equation 4.3]{Meerschaert2013}) and rewrite (\ref{eq:one dimensional FPE of A_E_t})
as
\[
\mathfrak{D}_{t}^{\alpha}p\left(x,t\right)=\Psi_{x}p\left(x,t\right)+\delta_{0}\left(dx\right)h\left(0^{+},t\right).
\]

\end{example}
Our last example concerns the one dimensional analogue of Corollary
\ref{cor:MD FPE of general A_E_t}. 
\begin{example}
Let $\left(A_{t},D_{t}\right)$ be a L\'{e}vy process as in Corollary
\ref{cor:MD FPE of general A_E_t}. Then its one dimensional distribution
$p\left(dx,t\right)$ satisfies 
\begin{equation}
\Xi_{x,t}p\left(dx,t\right)=\integral{r=0}{\infty}K\left(dx,dr+t\right).\label{eq:one dimensional FPE of uncoupled A_E_t}
\end{equation}
This was shown in \cite[Theorem 4.1]{Jurlewicz2012}. \end{example}
\begin{rem}
In \cite[Equation 5.9]{baule2005joint}, using different methods,
Baule and Friedrich essentially obtained Equation (\ref{eq:governing equation of multi variable h})
for the case where $D_{t}$ is a standard stable subordinator. In
\cite[Equation 14]{baule2007fractional} Baule and Friedrich obtained
Equation (\ref{eq:multidimensional FPE of A_E_t})(uncoupled case)
for the two dimensional case where $D_{t}$ is a standard stable subordinator. 
\end{rem}

\section{Directional Pseudo-Differential Operators \label{sec:Pseudo-Differential-Operators-on}}

In this section we wish to give a meaning to the PDO $\Psi_{\mathbf{x}}$,$\Phi_{\mathbf{t}}$
and $\Xi_{\mathbf{x},\mathbf{t}}$ discussed earlier. We shall see
that they are directional versions of their one-dimensional counterparts
$\Psi_{x}$,$\Phi_{t}$ and $\Xi_{x,t}$. We shall focus on $\Xi_{\mathbf{x},\mathbf{t}}$
as it is a generalization of $\Psi_{\mathbf{x}}$,$\Phi_{\mathbf{t}}$.
To illustrate the kind of results we are looking for, let us look
at the next simple example. Assume we have the following equation
in $\mathbb{\mathbb{R}}^{2}$
\begin{equation}
\left(\frac{\partial}{\partial x_{1}}+\frac{\partial}{\partial x_{2}}\right)f\left(x_{1},x_{2}\right)=h\left(x_{1},x_{2}\right).\label{eq:directional PDE}
\end{equation}
By using the change of variables $\left(x_{1},x_{2}\right)^{T}=\mathcal{T}\left(x'_{1},x'_{2}\right)^{T}$
where $\mathcal{T}$=$\left(\begin{array}{cc}
1 & 0\\
1 & 1
\end{array}\right)$ we can rewrite equation (\ref{eq:directional PDE}) as
\[
\frac{\partial}{\partial x'_{1}}f\left(\mathcal{T}\mathbf{x}'\right)=h\left(\mathcal{T}\mathbf{x}'\right).
\]
If we think of the change of variables $\mathcal{T}$ as an operator
on functions, i.e. $\mathcal{T}f:=f\left(\mathcal{T}\mathbf{x}\right)$
we see that
\[
\left(\frac{\partial}{\partial x_{1}}+\frac{\partial}{\partial x_{2}}\right)=\mathcal{T}^{-1}\frac{\partial}{\partial x'_{1}}\mathcal{T}.
\]
Since the operator $\mathbb{D}_{\mathbf{x}}$ is the classic one-dimensional
derivative under the change of variables $\mathcal{T}$ we say that
it is a directional version of the classic derivative. We wish to
show a similar result, i.e. that if $\xi\left(-k,s\right)$ is a L\'{e}vy
symbol, and therefore a symbol of a one-dimensional PDO $\Xi_{x,t}$,
then $\xi\left(-\sum_{i=1}^{n}k_{i},\sum_{i=1}^{n}s_{i}\right)$ is
the symbol of a PDO $\Xi_{\mathbf{x},\mathbf{t}}$ that is a directional
version of $\Xi_{x,t}$. \\
In \cite{baeumer2005space}, the authors showed that if $\xi\left(k,s\right)$
is a L\'{e}vy symbol then it is the symbol of a PDO on a Banach space.
More precisely, Let $X=L_{\omega}^{1}\left(\mathbb{R}\times\mathbb{R}_{+}\right)$
be the space of measurable functions s.t $\left\Vert f\right\Vert _{\omega}=\integral{\mathbb{R}}{}\integral{\mathbb{R}_{+}}{}\left|f\left(x,t\right)\right|e^{-\omega t}dtdx<\infty$
where $\omega>0$ is fixed. With this norm, the space $X$ is a Banach
space and the FLT is defined for each $f\in X$ for $k\in\mathbb{R},s\in(\omega,\infty)$.
Let $\xi\left(k,s\right)$ be a L\'{e}vy symbol, then it was shown that
$\xi\left(-k,s\right)$ is the symbol of the generator $L$ of a Feller
semigroup on $X$. Moreover, the domain of $L$ is given by 
\[
D\left(L\right)=\left\{ f\in X:\xi\left(-k,s\right)\overline{f}\left(k,s\right)=\overline{h}\left(k,s\right),\exists h\in X\right\} .
\]
\\
Let $L_{\bm{\omega}}^{1}\left(\mathbb{R}^{n}\times\mathbb{R}^{n}\right)$
denote the space of measurable functions that are defined on $\mathbb{R}^{n}\times\mathbb{R}^{n}$
and 
\[
\integral{\mathbb{R}^{n}}{}\integral{\mathbb{R}^{n}}{}\left|f\left(\mathbf{x},\mathbf{t}\right)\right|e^{-\left\langle \bm{\omega},\mathbf{t}\right\rangle }d\mathbf{x}d\mathbf{t}<\infty,
\]
for some $\bm{\omega}\in\mathbb{R}_{+}^{n}$. Let $\mathcal{A}_{n}\subset L_{\bm{\omega}}^{1}\left(\mathbb{R}^{n}\times\mathbb{R}^{n}\right)$
be the set of functions that vanish outside $\mathbb{R}^{n}\times A^{n}$
where $A^{n}=\left\{ \mathbf{t}:0<t_{1}\leq t_{2}\leq\ldots\leq t_{n}<\infty\right\} $.
Note that $\mathcal{A}_{n}$ is itself a Banach space and that the
FLT of $f\in\mathcal{A}_{n}$ is defined for $\mathbf{k}\in\mathbb{R}^{n}$,
$\mathbf{s}\in\hat{A}^{n}:=\left\{ \mathbf{s}:\omega_{i}<s_{i}\right\} $.
Let $\mathcal{T}$ be the element of $\GL_{n}$($n\times n$ invertible
matrices) s.t its first column is $\left(1,1,...,1\right)^{T}$ and
its $i$'th column $c_{i}\left(j\right)$ is $\delta_{i}\left(j\right)$(1
if $i=j$ and zero otherwise) for $1<i\leq n$. Note that $\det\left(\mathcal{T}\right)=1$
so that $\mathcal{T}$ is a bijection from the set $B^{n}=\left\{ \left(x_{1},x_{2},...,x_{n}\right):x_{1}>0,0\leq x_{2}\leq x_{3}\leq...\leq x_{n}\right\} $
onto $A^{n}$. We introduce the change of variables $\mathcal{T}\mathbf{x}'=\mathbf{x}$
and $\mathcal{T}\mathbf{t}'=\mathbf{t}$ and see that for $f\in\mathcal{A}_{n}$
\begin{alignat}{1}
\integral{\mathbb{R}^{n}}{}\integral{\mathbb{R}_{+}^{n}}{}\left|f\left(\mathbf{x},\mathbf{t}\right)\right|e^{-\left\langle \bm{\omega},\mathbf{t}\right\rangle }d\mathbf{x}d\mathbf{t} & =\integral{\mathbb{R}^{n}}{}\integral{\mathbb{R}_{+}^{n}}{}\left|f\left(\mathcal{T}\mathbf{x}',\mathbf{\mathcal{T}\mathbf{t}}'\right)\right|e^{-\left\langle \bm{\omega},\mathcal{T}\mathbf{t}'\right\rangle }det(\mathcal{T})^{2}d\mathbf{x}'d\mathbf{t}'\label{eq:Isometry}\\
 & =\integral{\mathbb{R}^{n}}{}\integral{\mathbb{R}_{+}^{n}}{}\left|f\left(\mathcal{T}\mathbf{x}',\mathbf{\mathcal{T}\mathbf{t}}'\right)\right|e^{-\left\langle \mathcal{T}^{*}\bm{\omega},\mathbf{t}'\right\rangle }d\mathbf{x}'d\mathbf{t}',\nonumber 
\end{alignat}
where $\mathcal{T}^{*}$is the adjoint of $\mathcal{T}$. Note that
$\mathcal{T}^{*}$ is the matrix whose first row is $\left(1,1,...,1\right)$
and its $i$'th row $r_{i}\left(j\right)$ is $\delta_{i}\left(j\right)$
for $1<i\leq n$. Let $\mathcal{B}_{n}\subset L_{\mathcal{T}^{*}\bm{\omega}}^{1}\left(\mathbb{R}^{n}\times\mathbb{R}^{n}\right)$
be the subspace of functions that vanish outside $\mathbb{R}^{n}\times B^{n}$
and note that it is a Banach space w.r.t the norm $\left\Vert f\right\Vert _{\mathcal{T}^{*}\bm{\omega}}$
and that its FLT is defined for \textbf{$\left(\mathbf{k};\mathbf{s}\right)\in\mathbb{R}^{n}\times\hat{B}^{n}$}
where $\hat{B}^{n}:=\mathcal{T}^{*}\hat{A}^{n}$. We can now define
the operator $\mathcal{T}:\mathcal{A}_{n}\rightarrow\mathcal{B}_{n}$
by $\left(\mathcal{T}f\right)\left(\mathbf{x}',\mathbf{t}'\right)=f\left(\mathcal{T}\mathbf{x}',\mathcal{T}\mathbf{t}'\right)$.
Note that by (\ref{eq:Isometry}) $\mathcal{T}$ is an isometric isomorphism
from $\mathcal{A}_{n}$ to $\mathcal{B}_{n}$ and that if $fg\in\mathcal{A}_{n}$
then $\mathcal{T}\left(fg\right)=\mathcal{T}\left(f\right)\mathcal{T}\left(g\right)\in\mathcal{B}_{n}$.
We abuse notation and define the operator $\Xi_{x,t}:\mathcal{B}^{n}\rightarrow\mathcal{B}^{n}$
by $f\left(\mathbf{x};\mathbf{t}\right)\mapsto\Xi_{x,t}f\left(\cdot,x_{2},...,x_{n};\cdot,t_{2},...,t_{n}\right)$.
Finally, we define the operator $\Xi_{\mathbf{x},\mathbf{t}}:\mathcal{A}^{n}\rightarrow\mathcal{A}^{n}$
by $\Xi_{\mathbf{x},\mathbf{t}}=\mathcal{T}^{-1}\Xi_{x,t}\mathcal{T}$.
\begin{prop}
\label{prop: Multi-dimensional PDO}Let $\xi\left(-k,s\right)$ be
a L\'{e}vy symbol of the on dimensional PDO $\Xi_{x,t}$, then $\Xi_{\mathbf{x},\mathbf{t}}$
is a PDO with symbol $\xi\left(-\sum_{i=1}^{n}k_{i},\sum_{i=1}^{n}s_{i}\right)$
and its domain is 
\begin{equation}
D\left(\Xi_{\mathbf{x},\mathbf{t}}\right)=\left\{ f\in\mathcal{A}_{n}:\xi\left(-\sum_{i=1}^{n}k_{i},\sum_{i=1}^{n}s_{i}\right)\overline{f}\left(\mathbf{k},\mathbf{s}\right)=\overline{h}\left(\mathbf{k},\mathbf{s}\right),\exists h\in\mathcal{A}_{n}\right\} .\label{eq:Domain of directional version PDO}
\end{equation}
\end{prop}
\begin{proof}
. By (\ref{eq:Isometry}) and Fubini's Theorem we have
\begin{equation}
\integral{\mathbb{R}}{}\integral{\mathbb{R}_{+}}{}\left|\mathcal{T}f\left(x_{1},x_{2},...,x_{n};t_{1},...,t_{n}\right)\right|e^{-\left\langle \mathcal{T}^{*}\bm{\omega},\left(t_{1},t_{2},...,t_{n}\right)\right\rangle }dx'_{1}dt'_{1}<\infty,\label{eq:In L_omega}
\end{equation}
and we see that $\mathcal{T}f\left(\cdot,x_{2},...,x_{n};\cdot,t_{2},...,t_{n}\right)\in L_{\sum_{i=1}^{n}\omega_{i}}^{1}\left(\mathbb{R}\times\mathbb{R}\right)$
for almost every $\left(x_{2},...,x_{n};t_{2},...,t_{n}\right)\in\mathbb{R}^{n-1}\times\mathbb{R}^{n-1}$.
On the other hand, introducing $\mathcal{T}^{*}\mathbf{k}=\mathbf{k}'$
and $\mathcal{T}^{*}\mathbf{s}=\mathbf{s}'$ %
{} for $\left(\mathbf{k}';\mathbf{s}'\right)\in\mathbb{R}^{n}\times\hat{B}^{n}$
and $f\in\mathcal{A}_{n}$ we have, 
\begin{alignat*}{1}
\overline{f}\left(\left(\mathcal{T}^{*}\right)^{-1}\mathbf{\mathbf{k}}',\left(\mathcal{T}^{*}\right)^{-1}\mathbf{s}'\right) & =\integral{\mathbb{R}^{n}}{}\integral{\mathbb{R}_{+}^{n}}{}e^{-i\left\langle \left(\mathcal{T}^{*}\right)^{-1}\mathbf{k'},\mathbf{x}\right\rangle -\left\langle \left(\mathcal{T}^{*}\right)^{-1}\mathbf{s'},\mathbf{t}\right\rangle }f\left(\mathbf{x},\mathbf{t}\right)d\mathbf{x}d\mathbf{t}\\
 & =\integral{\mathbb{R}^{n}}{}\integral{\mathbb{R}_{+}^{n}}{}e^{-i\left\langle \mathbf{k'},\left(\mathcal{T}^{-1}\right)\mathbf{x}\right\rangle -\left\langle \mathbf{s'},\left(\mathcal{T}^{-1}\right)\mathbf{t}\right\rangle }f\left(\mathbf{x},\mathbf{t}\right)d\mathbf{x}d\mathbf{t}\\
 & =\integral{\mathbb{R}^{n}}{}\integral{\mathbb{R}_{+}^{n}}{}e^{-i\left\langle \mathbf{k'},\mathbf{x}'\right\rangle -\left\langle \mathbf{s'},\mathbf{t}'\right\rangle }f\left(\mathcal{T}\mathbf{x}',\mathcal{T}\mathbf{t}'\right)det(\mathcal{T}^{-1})^{2}d\mathbf{x}'d\mathbf{t}'\\
 & =\integral{\mathbb{R}^{n}}{}\integral{\mathbb{R}_{+}^{n}}{}e^{-i\left\langle \mathbf{k'},\mathbf{x}'\right\rangle -\left\langle \mathbf{s'},\mathbf{t}'\right\rangle }\mathcal{T}f\left(\mathbf{x}',\mathbf{t}'\right)d\mathbf{x}'d\mathbf{t}'.
\end{alignat*}
It follows that for $f\in\mathcal{A}_{n}$ on $\mathbb{R}^{n}\times\hat{B}^{n}$
we have 
\begin{alignat*}{1}
\left(\mathcal{T}^{*}\right)^{-1}\overline{f} & =\overline{\mathcal{T}f}.
\end{alignat*}
To summarize, we have shown that the following diagram is commutative.
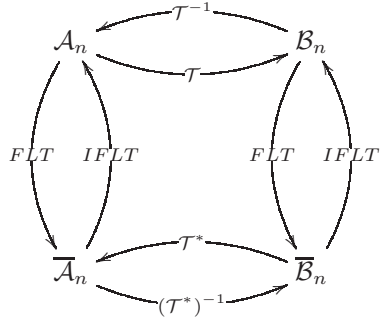
\[
\xyR{6pc}\xyC{6pc}\xymatrix{\mathcal{A}_{n}\ar@/_{1pc}/[r]|-{\mathcal{T}}\ar@/_{1.25pc}/[d]|-{FLT} & \mathcal{B}_{n}\ar@/_{1pc}/[l]|-{\mathcal{T}^{-1}}\ar@/_{1.25pc}/[d]|-{FLT}\\
\overline{\mathcal{A}}_{n}\ar@/_{1pc}/[r]|-{\left(\mathcal{T}^{*}\right)^{-1}}\ar@/_{1.25pc}/[u]|-{IFLT} & \overline{\mathcal{B}}_{n}\ar@/_{1pc}/[l]|-{\mathcal{T}^{*}}\ar@/_{1.25pc}/[u]|-{IFLT}
}
\]
Here $\overline{\mathcal{A}}_{n}$ and $\overline{\mathcal{B}}_{n}$
denote the image of $\mathcal{A}_{n}$ and $\mathcal{B}_{n}$ respectively
under the FLT map and IFLT is the inverse FLT. Note that $\overline{\mathcal{A}}_{n}$
is defined on $\mathbb{R}^{n}\times\hat{A}^{n}$ while $\overline{\mathcal{B}}_{n}$
is defined on $\mathbb{R}^{n}\times\hat{B}^{n}$. Next, we note that
the domain of $\Xi_{x',t'}$ on $\mathcal{B}_{n}$ is 
\[
D\left(\Xi_{x',t'}\right)=\left\{ f:\xi\left(-k'_{1},s'_{1}\right)f\left(\mathbf{k}',\mathbf{s}'\right)=h\left(\mathbf{k}',\mathbf{s}'\right),\exists h\in\mathcal{B}_{n}\right\} .
\]
Indeed, if $f\in\mathcal{B}_{n}$ satisfies $\xi\left(-k'_{1},s'_{1}\right)f\left(\mathbf{k}',\mathbf{s}'\right)=h\left(\mathbf{k}',\mathbf{s}'\right)$
for some $h\in\mathcal{B}_{n}$ then by (\ref{eq:In L_omega}) and
the results in \cite{baeumer2005space} we see that $f\in D\left(\Xi_{x,t}\right)$.
Next we show that $\Xi_{\mathbf{x},\mathbf{s}}$ is a PDO with symbol
$\psi\left(-\sum_{i=1}^{n}k_{i},\sum_{i=1}^{n}s_{i}\right)$. Suppose
$f\in D\left(\Xi_{\mathbf{x},\mathbf{s}}\right)$, then
\begin{equation}
\Xi_{x',t'}\mathcal{T}f\left(\mathbf{x}',\mathbf{s}'\right)=\mathcal{T}h\left(\mathbf{x}',\mathbf{s}'\right),\label{eq:symbol proof1}
\end{equation}
for some $h\in\mathcal{A}_{n}$. Applying FLT on both sides we obtain%
\begin{equation}
\psi\left(-k'_{1},s'_{1}\right)\left(\mathcal{T}^{*}\right)^{-1}\overline{f}\left(\mathbf{k}',\mathbf{s}'\right)=\left(\mathcal{T}^{*}\right)^{-1}\overline{h}\left(\mathbf{k}',\mathbf{s}'\right),\label{eq:symbol proof2}
\end{equation}
multiplying both sides by $\mathcal{T}^{*}$ we have
\begin{equation}
\psi\left(-\sum_{i=1}^{n}k_{i},\sum_{i=1}^{n}s_{i}\right)\left[\mathcal{T}^{*}\left(\mathcal{T}^{*}\right)^{-1}\overline{f}\left(\mathbf{k},\mathbf{s}\right)\right]=\overline{h}\left(\mathbf{k},\mathbf{s}\right).\label{eq:symbol proof3}
\end{equation}
Hence, $\Xi_{\mathbf{x},\mathbf{s}}$ is a PDO with symbol $\psi\left(-\sum_{i=1}^{n}k_{i},\sum_{i=1}^{n}s_{i}\right)$.
It is left to show that the domain of $\Xi_{\mathbf{x},\mathbf{s}}$
is as in (\ref{eq:Domain of directional version PDO}). Since $\mathcal{T}$
is a bijection it is clear that the following bijection holds $\mathcal{T}D\left(\Xi_{\mathbf{x},\mathbf{t}}\right)=D\left(\Xi_{x,t}\right)$
and the claim can be seen to be true through Equations (\ref{eq:symbol proof1}),
(\ref{eq:symbol proof2}) and (\ref{eq:symbol proof3}).
\end{proof}
In order to give a meaning to Equations (\ref{eq:multidimensional FPE of A_E_t})
and (\ref{eq:multidimensional FPE of general A_E_t}) through Proposition
\ref{prop: Multi-dimensional PDO} we define the function
\[
f\left(\mathbf{x};\mathbf{t}\right)=\integral{\mathbb{R}^{n}}{}g\left(\mathbf{x}-\mathbf{y}\right)p\left(d\mathbf{y};\mathbf{t}\right),
\]
where $g$ is a smooth function with compact support in $\mathbb{R}^{n}$
and $p\left(d\mathbf{x};\mathbf{t}\right)$ is a parameterized distribution
as in Subsection \ref{sub:Notations}. It follows that $f\left(\mathbf{x};\mathbf{t}\right)$
is smooth in $\mathbb{R}^{n}$ for every $\mathbf{t}\in\mathbb{R}_{+}^{n}$.
Multiply both sides of Equation (\ref{eq:multidimensional FPE of general A_E_t FLT})
by $\widetilde{g}\left(\mathbf{k}\right)$ and use the convolution-multiplication
property of the FT to obtain 
\[
\Xi_{\mathbf{x},\mathbf{t}}f\left(\mathbf{x};\mathbf{t}\right)=\integral{\mathbb{R}^{n}}{}g\left(\mathbf{x}-\mathbf{y}\right)p_{0}\left(d\mathbf{y};\mathbf{t}\right),
\]
where 

\begin{alignat*}{1}
p_{0}\left(d\mathbf{x};\mathbf{t}\right) & =\integral{r_{1}=0}{\infty}K\left(dx_{1},dr_{1}+t_{1}\right)\\
 & \phantom{}\times Q_{t_{2}-t_{1}}\left(x_{1},r_{1};dx_{2},dr_{2}\right)\circ\cdots Q_{t_{n}-t_{n-1}}\left(x_{n-1},r_{n-1};dx_{n},dr_{n}\right)\circ.
\end{alignat*}
This interpretation of (\ref{eq:multidimensional FPE of A_E_t}) and
(\ref{eq:multidimensional FPE of general A_E_t}) is in the spirit
of that in \cite[Chpater 4]{pazy2012semigroups} and was used in \cite[p. 15]{busani2014}.

\section{Conclusions}

In this paper we find the FDD's FFPEs of the process $A_{E_{t}}$
where $\left(A_{t},D_{t}\right)$ is a L\'{e}vy process, $D_{t}$ is a
strictly increasing subordinator with no drift and $E_{t}$ is the
inverse of $D_{t}$. The general form of these FFPEs (Eq. \ref{eq:multidimensional FPE of general A_E_t})
is a PDO in time and space variables applied to the distribution of
the process on one side of the equation while on the other we have
a term that accounts for the portion of particles that yet to be mobilized.
Moreover, considering the difference between the RL derivative and
that of Caputo's in the one dimensional case, and compared to the
finite dimensional one, it seems that the RL derivative is more suitable
in the context of CTRWL. We also showed that the PDO which appear
in Corollary \ref{cor:MD FPE of general A_E_t} are indeed bona fide
PDO and in fact a directional version of their one dimensional counterparts.

\section*{\textbf{Acknowledgment}}

The author would like to thank Mark M. Meerschaert of the Department
of Statistics and Probability at Michigan State University for suggesting
the problem of finding the finite dimensional FFPEs of CTRWL and introducing
the author to his work on the subject. 

\newpage{}

\bibliographystyle{plain}
\bibliography{FDFPEref}

\begin{thebibliography}{10}

\bibitem{baeumer2005space}
Boris Baeumer, Mark Meerschaert, and Jeff Mortensen.
\newblock Space-time fractional derivative operators.
\newblock {\em Proceedings of the American Mathematical Society},
  133(8):2273--2282, 2005.

\bibitem{baule2005joint}
Adrian Baule and Rudolf Friedrich.
\newblock Joint probability distributions for a class of non-{M}arkovian
  processes.
\newblock {\em Physical Review E}, 71(2):026101, 2005.

\bibitem{baule2007fractional}
Adrian Baule and Rudolf Friedrich.
\newblock A fractional diffusion equation for two-point probability
  distributions of a continuous-time random walk.
\newblock {\em EPL (Europhysics Letters)}, 77(1):10002, 2007.

\bibitem{becker2004limit}
Peter Becker-Kern, Mark~M Meerschaert, and Hans-Peter Scheffler.
\newblock Limit theorems for coupled continuous time random walks.
\newblock {\em Annals of Probability}, pages 730--756, 2004.

\bibitem{busani2014}
Ofer Busani.
\newblock Aging of uncoupled continuous time random walk limits.
\newblock {\em To appear in {E}lectronic {J}ournal of {P}robability}, 2015.

\bibitem{Jurlewicz2012}
Agnieszka Jurlewicz, P~Kern, Mark~M Meerschaert, and H-P Scheffler.
\newblock Fractional governing equations for coupled random walks.
\newblock {\em Computers \& Mathematics with Applications}, 64(10):3021--3036,
  2012.

\bibitem{Meerschaert2004}
Mark~M Meerschaert, Hans-Peter Scheffler, et~al.
\newblock Limit theorems for continuous-time random walks with infinite mean
  waiting times.
\newblock {\em Journal of applied probability}, 41(3):623--638, 2004.

\bibitem{Meerschaert2011a}
Mark~M Meerschaert and Alla Sikorskii.
\newblock {\em Stochastic models for fractional calculus}, volume~43.
\newblock Walter de Gruyter, 2011.

\bibitem{Meerschaert2013}
Mark~M Meerschaert and Peter Straka.
\newblock Inverse stable subordinators.
\newblock {\em Mathematical modelling of natural phenomena}, 8(02):1--16, 2013.

\bibitem{Meerschaert}
Mark~M Meerschaert, Peter Straka, et~al.
\newblock Semi-{M}arkov approach to continuous time random walk limit
  processes.
\newblock {\em The Annals of Probability}, 42(4):1699--1723, 2014.

\bibitem{pazy2012semigroups}
Amnon Pazy.
\newblock {\em Semigroups of linear operators and applications to partial
  differential equations}, volume~44.
\newblock Springer Science \& Business Media, 2012.

\bibitem{sato1999levy}
Ken-Iti Sato.
\newblock {\em L{\'e}vy processes and infinitely divisible distributions},
  volume~68.
\newblock Cambridge university press, 1999.

\bibitem{Straka2011diss}
Peter Straka.
\newblock {\em Continuous Time Random Walk Limit Processes - Stochastic Models
  for Anomalous Diffusion}.
\newblock Phd dissertation, University of New South Wales, 2011.

\bibitem{straka2011}
Peter Straka and Bruce~Ian Henry.
\newblock Lagging and leading coupled continuous time random walks, renewal
  times and their joint limits.
\newblock {\em Stochastic Processes and their Applications}, 121(2):324--336,
  2011.

\end{thebibliography}

\end{document}